\documentclass[12pt, reqno]{amsart}

\title[Nonexistence of global solutions]{Nonexistence of solutions of certain semilinear heat equations}

\emergencystretch=\maxdimen
\author[D. Suragan and B. Talwar]{Durvudkhan Suragan and Bharat Talwar}
\address{Durvudkhan Suragan, Department of Mathematics, Nazarbayev University, Astana 010000, Kazakhstan}
\email{durvudkhan.suragan@nu.edu.kz}
\address{Bharat Talwar, Department of Mathematics, Indian Institute of Science Education and Research, Bhopal 462066, Madhya Pradesh, India}
\email{btalwar.math@gmail.com, bharattalwar@iiserb.ac.in}

\allowdisplaybreaks[1]

\usepackage{relsize}
\usepackage{amsmath,amsfonts,amssymb,amsthm,mathrsfs}
\usepackage{hyperref}
\usepackage{cleveref}
\usepackage[margin=2.1 cm]{geometry}
\usepackage{setspace}
\usepackage{mathtools}  

\usepackage{setspace}
\usepackage{xcolor}

\allowdisplaybreaks[1]

\numberwithin{equation}{section}
\newtheorem{theorem}{\bf Theorem}[section]

\newtheorem{cor}[theorem]{\bf Corollary}
\newtheorem{remark}[theorem]{\bf Remark}

\newtheorem{defin}[theorem]{\bf Definition}

\newcommand{\dive}{\text{div}}
\newcommand{\seq}{\subseteq}

\newcommand{\N}{\mathbb{N}}

\newcommand{\R}{\mathbb{R}}

\newcommand{\ol}{\overline}

\newcommand{\norm}[1]{\| #1 \|}

\makeatletter \@namedef{subjclassname@2020}{\textup{2020} Mathematics Subject Classification} \makeatother

\begin{document}
	\begin{abstract}
		We consider a semilinear heat equation involving a forcing term which depends only on the space variable.
		To start with, the existence of a local mild solution is proved through an application of the Banach fixed-point theorem.
		With the help of carefully defined test functions, we then prove the nonexistence of global weak solutions.
		The most crucial step is to find the function $d(x)$ used in our proofs, which seems to depends only upon the considered vector fields.
		This leads to lower bounds for a possible critical Fujita-type exponent.
		The same function $d(x)$ could lead to a potential norm function which would be most suitable while working with these vector fields. 
		\Cref{Section4} is the attraction of this paper in which we apply our approach to all of the vector fields discussed by Biagi, Bonfiglioli and Bramanti in \cite{Biagi} giving rise to Grushin-type and Engel-type PDOs, and more.
		An upper bound for the blow-up time of local solutions is also provided in each of these cases.
	\end{abstract}
	\maketitle
	
	\section{Introduction}
	Let $X = (X_1, \dots, X_m)$ be a system of $C^\infty$ real vector fields on $\R^n$.
	We study the operator $$\Delta_X = - \sum_{k=1}^m X_k^\ast X_k,$$  where $X_k^\ast = -X_k - \dive X_k$.
	The operator $-\Delta_X$ is symmetric on $C_0^\infty(\R^n)$ and well-defined on $$\{ u \in H^1_{X,0}(\R^n): \Delta_X u \in L^2(\R^n) \},$$ where
	$H^1_{X,0}(\R^n)$ is the closure of $C_0^\infty(\R^n)$ in $H^1_{X}(\R^n) := \{ u \in L^2(\R^n): X_i u \in L^2(\R^n), 1 \leq i \leq m \}$ which is endowed with the norm given by the formula $$\norm{u}^2 = \int_{R^n}|u|^2(x) dx + \sum_{i=1}^m \int_{\R^n} |X_i u|^2(x) dx.$$

	For $p > 1$ and $n > 2$ (except when mentioned otherwise), we consider the problem of finding conditions for the non-existence of solutions for
	\begin{equation}\label{MainProblem}
		\begin{cases} 
			u_t(t,x) - \Delta_X u(t,x) = |u(t,x)|^p + f(x), & (t,x) \in (0,T) \times \R^n \\
			u(0, x) = u_0(x), & x \in \R^n,
		\end{cases}
	\end{equation}
	where the assumptions on the forcing term $f$ will be described as and when required in the due course.
	The complexity in finding the existence of a solution for such a problem is that while a time-dependent solution $u$ is sought, the function $u_t(t,x) - \Delta_X u(t,x) - |u(t,x)|^p$ is demanded to be independent of time.
	The most straightforward approach then is to seek a stationary solution i.e., one that remains unaffected by change in time.
	By doing so, the term $u_t(t,x)$ becomes irrelevant in the context of this problem.

	We focus solely on describing the conditions under which there is no global solution to the problem \ref{MainProblem}.
	A notion that becomes important for this goal is the notion of a critical exponent, which, if it exists, is the least real number $p_c$ such that for every $p > p_c$, the existence of a global solution is guaranteed for some $f$ and $u_0$; and a global solution never exists for $p \leq p_c$.

	The literature on similar problems is vast, originating from the seminal works \cite{Fujita} and \cite{Hayakawa, Kobayashi}, where Fujita and others considered the case $f \equiv 0$ and established that the critical exponent is $\frac{n+2}{n}$ when $X_k=\partial_{k} \text{ and } k=\overline{1,n}$.
	For non-zero $f$, Zhang \cite{Zhang} considered a system similar to (\ref{MainProblem})  on a Riemannian manifold ${\bf M}^n$ with possibly non-negative Ricci curvature.
	There $\frac{\alpha}{\alpha - 2}$ for some $\alpha$ was proved to be the critical exponent, where $\alpha \leq n$ in general and $\alpha = n$ when ${\bf M}^n = \R^n$.
	Bandle et al.~proved in \cite[Theorem 2.1]{Bandle}  that if one replaces $\Delta_X$ with Laplacian $\Delta$  in (\ref{MainProblem}), then the critical exponent is $\frac{n}{n-2}$.
	As $\frac{n}{n-2} > \frac{n+2}{n}$, it shows that one can choose $p$ from a bigger interval $(1, \frac{n}{n-2}]$ and still the nonexistence of a global solution is guaranteed when a suitable forcing term $f$ is involved.
	Another interesting phenomenon observed during similar studies is that a solution exists for some finite time and then it blows-up in the supremum norm, hindering the global existence.
	Determining the blow-up time remains a challenging question in the field.
	Ever since \cite{Fujita} appeared, numerous results concerning existence, nonexistence, and blow-up phenomena have emerged not only on $\R^n$
	but also on general topological groups. These investigations encompass studies on several domains -see \cite{Torebek, WangYinYou, RuzhanskyYessirkegenov, SunLiuLi, GeorgievPalmieri,JleliKawakamiSamet,Pohozaev,Pascucci,Bernard,Levine,BarasPierre} and references therein.

	While $\Delta_X$ is well-studied with an extra H\"ormander condition (H) on $X$ which is described below, we do not impose any such restriction.
	Nevertheless, this condition is met by a class of operators we investigate in \Cref{Section4}.
	It is noteworthy that the counterpart of $\Delta_X$ on bounded domains  was recently studied in \cite{ChenChen}, along with the H\"ormander condition.
	
	\smallskip
	
	\noindent (H): There exists a natural number $r$ such that the vector fields $X_1, \dots, X_m$ together with their $r$ commutators span the tangent space at each point of $\R^n$.
	
	\smallskip

	It is crucial to note that in exploring potential critical exponents for problems akin to (\ref{MainProblem}), existing literature extensively utilizes the geometry associated with the pair $(\R^n,X)$.
	However, the primary aim of this paper is to push the boundaries without relying on such geometric considerations.
	Consequently, we depart from the extensively studied class of vector fields and consider different systems of vector fields.
	While \Cref{Section3} serves to bolster confidence in our approach, the major attraction of this paper is \Cref{Section4} which addresses more complex and significant vector fields by appealing to the proof approach of \Cref{Section3}.
	Our techniques and results draw inspiration from \cite{Torebek, JleliKawakamiSamet} and choosing suitable function $d(x)$ is the crucial part of our proofs.
	The major attractions of this article are as follows.
	\begin{enumerate}
		\item The local existence result detailed in \Cref{LocalExistence}.
		\item Lower bounds for a possible critical exponents provided in \Cref{Section4}.
		\item An upper bound on the blow up time provided as different corollaries in Sections 3 $\&$ 4.
	\end{enumerate}

	Let us briefly summarize the layout and main results of this paper.	
	As an application of the Banach fixed point theorem,  in \Cref{LocalExistence} we prove the existence and blow-up of a local in time mild solution.	
	In \Cref{Section3} we work with the assumption that for $X_k = \sum_{i=1}^n a_{k,i} \frac{\partial}{\partial x_i}$ and $1 \leq k \leq m$, the functions $a_{k,i}$ and $\frac{\partial{a_{k,i}}}{\partial{x_j}}$ are bounded.
	This holds, for example, for sine, cosine and constant functions which are important from harmonic analysis point of view.
	For locally integrable function $f$ with finite and positive integration over the space, it is proved through Theorems  \ref{subcritical} and \ref{critical} that a lower bound for the possible critical exponent in (\ref{MainProblem}) is $\frac{n}{n-1}$.
	An upper bound on the time of blow-up is also provided in \Cref{BlowUpTime}.
	In \Cref{Section4}, we consider different vector fields, including all the vector fields $X$ discussed in \cite{Biagi}, which give rise to the so-called Grushin-type and Engel-type  partial differential operators.
	For details on the significance of these operators, we refer the reader to \cite{BiagiBonfoglioli, Bauer} and references therein.
	A noteworthy information about some of these operators is that they do not give rise to a stratified Lie group structure on the base spaces under consideration.

	Throughout the paper we assume the  existence (see \cite{VSC}) of a non-trivial and non-negative function $h_t(x,y)$ satisfying $\int_{\R^n} h_t(x,y) dy \leq 1$ and making $\{ e^{-t \Delta_X} \}_{t \geq 0}$ a semigroup with infinitesimal generator $-\Delta_X$, such that $$e^{-t \Delta_X} w(x) := \int_{\R^n} h_t(x,y) w(y) dy$$ for every $x \in \R^n$ and $w \in C_0(\R^n)$.
	As a consequence, we have $$\norm{e^{-t \Delta_X} w(x)} \leq  \int_{\R^n} h_t(x,y) \norm{w(y)} dy \leq \norm{w}_{\sup}$$ for every $x \in \R^n$ and $t > 0$.
	This property of the semigroup $\{ e^{-t \Delta_X} \}_{t \geq 0}$ will be crucial to the proof of \Cref{LocalExistence}.
	Note that the so-called heat-kernel satisfies these properties and its existence is guaranteed in many general cases, for example, \cite[Theorem 10.1]{Li} provides the existence on compact manifolds with boundary and \cite[Chapter 4]{VSC} gives existence on nilpotent Lie groups.
	On $\R^n$,  existence of heat-kernel for $\Delta_X - \frac{\partial}{\partial t}$ is mentioned in \cite{KusuokaStroock} and in \cite{Biagi2} global estimates for heat kernel are given for some of the  operators we study in \Cref{Section4}.

	\section{Local existence} \label{LocalExistenceSection}
	Let us first fix some notations.
	We denote the space of all continuous functions on $\R^n$ which vanish at infinity by $C_0(\R^n)$.
	For a Banach space $(Y, \norm{\cdot})$, the space of all $Y$-valued continuous and vanishing at infinity functions on a locally compact Hausdorff topological space $S$ is denoted by $C_0(S,Y)$.
	This space is a Banach space when norm of $w \in  C_0(S,Y)$ is given by $$\norm{w}_{\sup} = \sup \{ \norm{w(s)} : s \in S \}.$$
	Similarly, $C_c(S,Y)$ will denote the normed space of compactly supported $Y$-valued continuous functions on $S$.
	We set $x = (x_1, x_2, \dots, x_n)$ and likewise $x' = (x_1', x_2', \dots, x_n')$.

	\begin{defin}(Mild solution)
		A local mild solution of (\ref{MainProblem}) is a function $u \in C([0,T], C_0(\R^n))$ for which $$u(t) = e^{-t \Delta_X} u(0) + \int_{0}^t e^{-(t-s) \Delta_X} (|u(s)|^p + f) ds,$$ for any $ t \in [0,T)$.
		The function $u$ is called a global mild solution of (\ref{MainProblem}) if $T = + \infty$.
	\end{defin}

	\begin{theorem}\label{LocalExistence}
		Let $u_0, f \in C_0(\R^n)$. Then the following holds.
		\begin{enumerate}
			\item There exists a time $T \in (0, \infty)$ such that (\ref{MainProblem}) has a unique local mild solution. 
			Moreover, this solution is bounded by twice the maximum of the supremum norms of $u_0$ and $f$.
			\item The solution in (1) can be uniquely extended to a maximal interval $[0, T_{\max})$.
			If $T_{\max} \in \R$, then supremum norm of $u(t) \in C_0(\R^n)$ approaches infinity as $t$ tends to $T_{\max}$.
			\item If $u_0, f \in L^q(\R^n)$ where $q \in [1, \infty]$, then $$u \in C([0,T_{\max}), C_0(\R^n)) \cap C([0,T_{\max}), L^q(\R^n)).$$
		\end{enumerate}
	\end{theorem}
	\begin{proof}
		1. Let us define a metric space $$\Theta = \{ u \in C([0,T], C_0(\R^n)) : \norm{ u }_{\sup} \leq 2 \delta(u_0,f), u(0) = u_0 \}$$ where $\delta(u_0,f) = \max \{ \norm{u_0}_{\sup}, \norm{f}_{\sup} \}$.
		The distance between two elements of $\Theta$ is naturally given by the supremum norm of their difference. However, for convenience, we denote this distance between $u,v \in \Theta$ by $\norm{u-v}_\Theta$.
		As convergence in supremum norm implies pointwise convergence, $\norm{ \cdot }_\Theta$ defines a metric on $\Theta$ making it a complete metric space.
		Define $\Phi: \Theta \to \Theta$ such that $$\Phi(v)(t) = e^{-t \Delta_X} v(0) + \int_{0}^t e^{-(t-s) \Delta_X} (|v(s)|^p + f) ds.$$
		We show that for an appropriate $T > 0$, $\Phi$ is a well-defined contraction. Then the Banach fixed point theorem for complete metric spaces may be applied to obtain a unique function $u \in \Theta$ such that $\Phi(u) = u$ which will be the required solution.

		It may still happen that another mild solution exists outside of $\Theta$.
		To show that this is not the case, we make use of the Gronwall's inequality and apply it to the definition of mild solution.
		Recall that  the Gronwall's inequality claims that if $$\gamma(t) \leq \alpha(t) + \int_{0}^t \beta(s) \gamma(s) ds$$ for some continuous functions $\alpha, \beta$ and $\gamma$, then $$\gamma(t) \leq \alpha(t) + \int_{0}^t \alpha(s) \beta(s) \exp^{\int_{s}^t \beta(r) dr} ds.$$
		If $u, v$ are two mild solutions, then we can use $\alpha = 0$ and $\beta = C(p) \max \{\norm{u^{p-1}}, \norm{u^{p-1}} \}$ and $\gamma(t) = \norm{ u(t) - v(t)}$ in Gronwall's inequality to obtain that $u =  v$.
		
		To complete the proof of (1), we now proceed towards proving that $T$ is a well defined contraction on $\Theta$.
		
		(i) Well-defined:  We need to prove that $\Phi(v) \in \Theta$ whenever $v \in \Theta$.
		In other words, we must prove that for $v \in \Theta$ we have $\Phi(v) \in C([0,T], C_0(\R^n))$, $\norm{\Phi(v)} \leq 2 \delta(u_0,f)$ and $\Phi(v)(0) = u_0$.
		It is clear that $\Phi(v)$ is a continuous function as it is composition of several continuous functions.
		Also, $\Phi(v)(0) = v(0) = u_0$.
		As $e^{-t \Delta_X} w(\cdot) = \int_{\R^n} h_t(\cdot,y) w(y) dy$ on  $\R^n$ for every $t > 0$, we have  \begin{eqnarray*}
			\lVert\Phi(v)(t)\rVert_{\sup} &=& \left\lVert e^{-t \Delta_X} v(0) + \int_{0}^t e^{-(t-s) \Delta_X} (|v(s)|^p + f) ds \right\rVert_{\sup} \\
			&\leq & \norm{e^{-t \Delta_X} v(0)}_{\sup} + \int_{0}^t \norm{e^{-(t-s) \Delta_X} (|v(s)|^p + f)}_{\sup} ds \\
			&\leq & \norm{v(0)}_{\sup} + t \norm{ |v(s)|^p + f}_{\sup}.
		\end{eqnarray*}
		Thus, $$\norm{\Phi(v)}_{\sup} \leq \norm{v(0)}_{\sup} + t \norm{ |v(s)|^p}_{\sup} + t \norm{f}_{\sup} \leq (1+t)\delta(u_0,f)+ t 2^p \delta(u_0,f)^p$$ and one may choose $T$ small enough so that $$\norm{\Phi(v)}_{\sup} \leq  (1+T+ 2^{p-1}T\delta(u_0,f)^{p-1} )\delta(u_0,f) \leq 2 \delta(u_0,f).$$

		Next we prove that for every $t \in [0,T]$, we have $\Phi(v)(t) \in C_0(\R^n)$.
		Let $1 >\epsilon >0$ be given.
		For each fixed $t \in [0, T]$, as $f$ and $v(t)$ vanish at infinity, there exists a compact subset $K_{t, \epsilon} \seq \R^n$ such that for $x \in \R^n \setminus K_{t, \epsilon}$ both $|f(x)|$ and  $|v(t)(x)|$ are less than $\epsilon/2$.
		Using continuity of $v$, choose a neighbourhood $V_t$ of $t \in [0, T]$ such that $|v(s)(x)| \leq \epsilon$ for every $x \in \R^n \setminus K_{t, \epsilon}$ and $s \in V_t$.
		Since $[0,T]$ is compact, there exists $t_1, t_2, \dots, t_l \in [0,T]$ such that $[0,T] = \cup_{i=1}^l V_{t_i}$.
		Then for $x \in \R^n \setminus \cup_{i=1}^l K_{t_i, \epsilon}$, we have \begin{eqnarray*}
			|\Phi(v)(t)(x)| &=& \left|e^{-t \Delta_X} v(0)(x) + \int_{0}^t e^{-(t-s) \Delta_X} (|v(s)(x)|^p + f(x)) ds \right| \\
			& \leq&  \epsilon + T \epsilon^p + \int_{0}^t |e^{-(t-s) \Delta_X} f(x)| ds \\
			& \leq&  \epsilon + T \epsilon^p + T \epsilon \leq (1+ 2T) \epsilon.
		\end{eqnarray*}
		As $\cup_{i=1}^l K_{t_i, \epsilon}$ is compact, we have established that $\Phi(v)(t) \in C_0(\R^n)$ for every $t \in [0,T]$.

		(ii) Contraction: Let $u,v \in \Theta$.
		As $u(0) = v(0)$, for some $C(p)>0$ we have
		\begin{eqnarray*}
			\norm{\Phi(u)(t) - \Phi(v)(t)} &=&  \left\lVert e^{-t \Delta_X} u(0) - e^{-t \Delta_X} v(0) + \int_{0}^t e^{-(t-s) \Delta_X} (|u(s)|^p  -|v(s)|^p) ds \right\rVert \\
			&\leq&  C(p) \int_{0}^t \norm{e^{-(t-s) \Delta_X}|u(s)  -v(s)|(|u(s)|^{p-1}+ |v(s|^{p-1})} ds\\
			&\leq& C(p) 2^{p} \delta(u_0,f))^{p-1} \int_{0}^t \norm{u(s)  -v(s)}ds\\
			&\leq &  C(p) 2^{p} \delta(u_0,f)^{p-1} t  \norm{u-v}_\Theta.
		\end{eqnarray*}
		One can now choose $T$ small enough so that $\Phi$ is a contraction on $[0,T]$.
		This shows that a unique solution exists in $\Theta$.

		2. Existence of $[0, T_{\max})$ follows from the uniqueness part of (1).
		Set $u$ to be this unique solution.
		Let us now assume that $T_{\max}$ is finite and there exists $K > 0$ such that $\norm{u(t)}_\Theta \leq K$ for every $0 \leq t < T_{\max}$.
		By demonstrating the existence of a mild solution for (\ref{MainProblem}) in $[0, \tau)$, for some $\tau > T_{\max}$, we will contradict the maximality of $T_{\max}$.
		
		Fix $t^\ast \in (\frac{T_{\max}}{2}, T_{\max})$, $T < T_{\max}$ and define a complete metric space $$\Theta_1 = \{ v \in C([0,T], C_0(\R^n)) : \norm{ v }_{\sup} \leq 2 \delta(K,f),v(0) = u_{t^\ast} \},$$ where $\delta(K,f) = \max \{K, \norm{f}_{\sup} \}$ and norm is denoted by $\norm{ \cdot }_{\Theta_1}$.
		Then just as in (1), the operator $\Phi_1: \Theta_1 \to \Theta_1$ given by $$\Phi_1(v) = e^{-t \Delta_X} u_{t^\ast} + \int_{0}^t e^{-(t-s) \Delta_X} (|v(s)|^p + f) ds$$ is a well-defined contraction.
		From the Banach fixed point theorem, there exists a fixed point $v$ of $\Theta_1$.
		
		Notice that no matter which $t^\ast$ we start with, the choice of $T$ such that a fixed point of $\Phi_1$ exists in $\Theta_1$ does not change.
		This is because for $\Phi_1$ to be a contraction from $\Theta_1$ to itself, we need $T$ small enough so that $$\norm{\Phi_1(v)}_{\sup} \leq  (1+T+ 2^{p-1}\delta(K,f)^{p-1} )\delta(u_0,f) \leq 2 \delta(K,f)$$ and $$\norm{\Phi_1(u)(t) - \Phi_1(v)(t)} \leq  C(p) 2^{p} \delta(K,f)^{p-1} T  \norm{u-v}_{\Theta_1}$$ hold, and none of these bounds depend on $t^\ast$.
		In view of this, choose $t^\ast$ such that $T + t^\ast > T_{\max}$.
		Since $u(t^\ast) = v_0$, it can be easily verified that
		\begin{equation}
			\ol{u}(t) = \begin{cases} 
				u(t), & t \in [0,t^\ast],\\
				v(t- t^\ast), & t \in [t^\ast, t^\ast + T],
			\end{cases}
		\end{equation}
		defines a solution for (\ref{MainProblem}), contradicting the maximality of $T_{\max}$.

		(3) As $C_0(\R^n)  \seq L^\infty(\R^n)$, the result holds trivially for $q = \infty$.
		Let $q \in [1, \infty)$ and consider the complete metric space $\Theta_2$ which is  nothing but the intersection of  $$\{ v \in C([0,T_{\max}), C_0(\R^n))  : \norm{v}_{\sup}\leq 2 \delta(u_0,f) ,v(0) = u_{0} \} $$ and $$ \{ v \in C([0, T_{\max}), L^q(\R^n)) : \norm{ v }_{\sup} \leq 2 \delta_q(u_0,f)\},$$ where $\delta(u_0,f) = \max \{\norm{u_0}_{\sup}, \norm{f}_{\sup} \}$, $\delta_q(u_0,f) = \max \{\norm{u_0}_{L^q(\R^n)}, \norm{f}_{L^q(\R^n)} \}$ and norm is given by $$\norm{v }_{\Theta_2} = \norm{ v}_{C([0,T], C_0(\R^n))} + \norm{v}_{C([0, T_{\max}), L^q(\R^n))}.$$
		We prove that if $v \in C([0, T_{\max}), L^q(\R^n))$, then for every $t \in [0,T_{\max})$ we have $\Phi(v)(t) \in L^q(\R^n)$.
		Rest follows from the proof of (1).

		Recall that $$\Phi(v)(t) = e^{-t \Delta_X} v(0) + \int_{0}^t e^{-(t-s) \Delta_X} (|v(s)|^p + f) ds.$$
		Note that $e^{-t \Delta_X}$ is an operator from $L^q(\R^n)$ to $L^q(\R^n)$ since, by definition, $e^{-t \Delta_X}(v) = v \ast h_t$ is convolution of $v \in L^q(\R^n)$ with $h_t \in L^1(\R^n)$.
		As $u_0 \in L^q(\R^n)$ we obtain $e^{-t \Delta_X} v(0) \in  L^q(\R^n)$.
		Moreover, as $v(s) \in C_0(\R^n) \cap L^q(\R^n)$ we  have $|v(s)|^p \leq \norm{ v(s)^{p-1}}_{\sup}|v(s)|$.
		Using $f \in L^q(\R^n)$ gives $|v(s)|^p + f \in L^q(\R^n)$.
		Application of the operator $e^{-(t-s) \Delta_X}$ then gives an element of $L^q(\R^n)$.
		Being linear combination of elements of $L^q(\R^n)$ the function $\Phi(v)(t)$ belongs to $L^q(\R^n)$.
	\end{proof}

	\section{Global nonexistence}\label{Section3}
	\begin{defin}\label{DefinitionLocalWeakSolution}
		Let $f, u_0 \in L^1_{loc}(\R^n)$.
		Then $u \in L^p_{loc}((0,T), L^p_{loc}(\R^n))$ is called a local in time weak solution of (\ref{MainProblem}) if for every non-negative test function $\psi \in C^1_c( [0,T); C_c^2(\R^n) )$ we have \begin{eqnarray}
			\int_0^T \int_{\R^n} |u|^p \psi + \int_{\R^n} u_0 \psi(0, x) + \int_0^T \int_{\R^n} f \psi + \int_0^T \int_{\R^n} u \psi_t + \int_0^T \int_{\R^n} u \Delta_X \psi = 0.
		\end{eqnarray}
		When $T = \infty$, such $u$ is called global in time weak solution.
	\end{defin}

	Throughout this paper we fix $\frac{1}{p} + \frac{1}{p'} = 1$.
	We keep $C, C_1,$ and $C_2$ to denote the constants but their values may keep on changing from line to line without effecting the analysis.
	Using the self-adjointness of the operator $\Delta_X$, it can be proved using standard techniques that when $u_0 \in C_0(\R^n)$, every local in time mild solution belonging to the class $C([0,T], L^\infty(\R^n))$ is a local in time weak solution as well.

	In our first theorem, we provide a lower bound for the critical exponent as $\frac{n}{n-1}$.
	
	\begin{theorem}\label{subcritical}
		Let $p < \frac{n}{n-1}$ and $f, u_0 \in L^1_{loc}(\R^n)$ with  $u_0 \geq 0$.
		If $\int_{\R^n} f(x) dx > 0$ and the functions $a_{k,i}$ and $\frac{\partial a_{k,j}}{\partial x_i}$ are bounded  for every $1 \leq k \leq m$ and $1 \leq i,j \leq n$, then (\ref{MainProblem}) does not admit global in time weak solution.
	\end{theorem}
	\begin{proof}
		Let us suppose on the contrary that $u$ is a global in time weak solution of (\ref{MainProblem}).
		Then for every test function $\psi$, $$\int_0^T \int_{\R^n} |u|^p \psi + \int_{\R^n} u_0 \psi(0, x) + \int_0^T \int_{\R^n} f \psi + \int_0^T \int_{\R^n} u \psi_t + \int_0^T \int_{\R^n} u \Delta_X \psi = 0.$$
		As both $\psi$ and $u_0$ are non-negative, we obtain \begin{eqnarray}\label{Inequality}
			\int_{0}^T \int_{\R^n} f \psi + \int_{0}^T \int_{\R^n} |u|^p \psi &\leq& - \int_{0}^T \int_{\R^n} u \psi_t   - \int_{0}^T \int_{\R^n} u \Delta_X \psi \nonumber  \\
			& \leq &  \int_{0}^T \int_{\R^n} | u || \psi_t| + |u|| \Delta_X \psi|.
		\end{eqnarray}
		Set $$P_x:= \frac{x_1^2 + x_2^2 + \dots+ x_n^2}{T}.$$
		Let us define a test function with separated variables as $$\psi(t,x) = \Phi(P_x) \Phi(2t/T),$$ where \begin{equation*}
			\Phi(z) = \begin{cases} 
				1, & z \in [0,1],\\
				\searrow, & z \in [1, 2],\\
				0, & z \in (2, \infty),
			\end{cases}
		\end{equation*}
		is a function in $C^2_c(\R^+)$ satisfying the estimates $$ \int_{1 \leq \norm{x} \leq 2} \left|\frac{\left(\Phi'(\norm{x}) + \Phi''(\norm{x})\right)^p}{\Phi(\norm{x})}\right|^{\frac{1}{p-1}} < \infty \text{ and } \int_{0}^2   \left|\frac{\Phi'(t)^p}{\Phi(t)}  \right|^{\frac{1}{p-1}} dt < \infty.$$
		For example, we can take $\Phi(z) = \exp^{1- \frac{1}{1- (1,z)^4}}$ on $[1,2]$ as in \cite{Bandle}. 
		It is not difficult to verify that $\frac{(\Phi'(x) + \phi''(x))^p}{\Phi(x)}$ and $\frac{\Phi'(x)^p}{\Phi(x)}$ have no singularity on 1 or 2 and that these functions are continuous on $[1,2]$, and hence the required integrability follows.
		Note that $$\text{supp}(\psi) = \left\{ (t,x) \in \R^+ \times \R^n : 0 \leq 2t/T \leq 2, 0 \leq P_x \leq 2 \right\}.$$
		Now, $$\psi_t(x,t) = \frac{\partial}{\partial t}(\Phi(P_x) \Phi\left(\frac{2t}{T}\right) = \frac{2}{T}  \Phi(P_x) \Phi'\left(\frac{2t}{T}\right)$$ implies that $$\text{supp}(\psi_t) = \{ (t,x) \in \R^+ \times \R^n : 1 \leq 2t/T \leq 2, 0 \leq P_x \leq 2 \}$$ and  \begin{eqnarray*}
			\Delta_X \psi(t,x) &=& \Phi(2t/T) \Delta_X  \Phi(P_x)  \\
			&=& \Phi(2t/T)  \sum_{k=1}^m \left(  X_k^2 \Phi(P_x) + (\dive X_k) X_k \Phi(P_x) \right)\\
			&=& \Phi(2t/T) \sum_{k=1}^m \left( \Phi^{''}(P_x) (X_k(P_x))^2 + \Phi'(P_x) X_k^2(P_x) \right. \\
			& +& \left.  \left(\sum_{i=1}^n \frac{\partial a_{k,i}}{\partial x_i}\right) \left(\sum_{j=1}^n a_{k,j} \frac{\partial}{\partial x_j} \Phi(P_x) \right) \right)\\
			&=& \Phi(2t/T)  \sum_{k=1}^m \left( \Phi^{''}(P_x) \left( \sum_{j=1}^n a_{k,j} \frac{\partial P_x}{\partial x_j} \right)^2 + \left(\sum_{i=1}^n \frac{\partial a_{k,i}}{\partial x_i} \right) \left(\sum_{j=1}^n a_{k,j} \Phi'(P_x) \frac{\partial P_x}{\partial x_j}  \right)   \right. \\
			& + & \left. \Phi'(P_x) \left( \sum_{i,j = 1}^n a_{k,i} a_{k,j} \frac{\partial^2 P_x }{\partial x_i \partial x_j} + \sum_{j=1}^n \frac{\partial P_x}{\partial x_j} \left( \sum_{i=1}^n a_{k,i} \frac{\partial a_{k,j}}{\partial x_i} \right)  \right) \right)\\
			&=& \Phi(2t/T)  \sum_{k=1}^m \left( \Phi^{''}(P_x) \left( \sum_{j=1}^n a_{k,j} \frac{2 x_j}{T} \right)^2 + \left(\sum_{i=1}^n \frac{\partial a_{k,i}}{\partial x_i} \right) \left(\sum_{j=1}^n a_{k,j} \Phi'(P_x) \frac{2 x_j}{T}  \right) \right. \\
			& + & \left. \Phi'(P_x) \left( \sum_{j = 1}^n a_{k,j}^2 \frac{2 }{T}  + \sum_{j=1}^n \frac{2 x_j}{T}  \left( \sum_{i=1}^n a_{k,i} \frac{\partial a_{k,j}}{\partial x_i} \right)  \right)   \right)
		\end{eqnarray*}
		implies that 
		$$\text{supp}(\Delta_X \psi) = \{ (t,x) \in \R^+ \times \R^n : 0 \leq 2t/T \leq 2, 1 \leq P_x \leq 2 \}.$$
		As $\Phi$ us decreasing in $[1,2]$, $\psi$ is zero whenever $\psi_t$ or $\Delta_X \psi$ is zero.
		Hence we can write 	$|u||\psi_t|$ as $|u| |\psi|^{\frac{1}{p}} |\psi|^{\frac{-1}{p}} |\psi_t|$ and $|u||\Delta_X \psi|$ as $|u| |\psi|^{\frac{1}{p}} |\psi|^{\frac{-1}{p}} |\Delta_X \psi|$.
		Applying $\frac{p}{2}-$Young's inequality, we obtain $$\int_{0}^T \int_{\R^n}|u||\psi_t|  \leq \frac{1}{2} \int_{0}^T \int_{\R^n}|u|^p|\psi| + \frac{1}{p' (\frac{p}{2})^{p'-1}} \int_{0}^T \int_{\R^n}|\psi|^{\frac{-1}{p-1}}|\psi_t|^{\frac{p}{p-1}}$$
		and 
		$$\int_{0}^T \int_{\R^n}|u||\Delta_X \psi|  \leq \frac{1}{2} \int_{0}^T \int_{\R^n}|u|^p|\psi| + \frac{1}{p' (\frac{p}{2})^{p'-1}} \int_{0}^T \int_{\R^n}|\psi|^{\frac{-1}{p-1}}|\Delta_X \psi|^{\frac{p}{p-1}}.$$
		Using inequality 3.2 we now obtain
		\[\int_{0}^T \int_{\R^n} f \psi \leq \frac{1}{p' (\frac{p}{2})^{p'-1}} \Big( \int_{0}^T \int_{\R^n} |\psi|^{\frac{-1}{p-1}}|\Delta_X \psi|^{\frac{p}{p-1}} + |\psi|^{\frac{-1}{p-1}}|\psi_t|^{\frac{p}{p-1}} \Big).\]

		In the support of $\Delta_X \psi$, we have $|x_i| \leq \sqrt{2T}$ for every $1 \leq i \leq n$.
		Using the assumptions on $a_{k,i}$ and the fact that $\Phi$ is twice continuously differentiable, one can choose \[M = \max  \left\{  \norm{a_{k,i}(x)}_{\sup}, \norm{a_{k,i}^2(x)}_{\sup}, \left\| \frac{\partial a_{k,j}(x)}{\partial x_i}\right\|_{\sup} \right\}_{1 \leq k \leq m; 1 \leq i,j \leq n}.\]
		Thus, for large $T$ there exists a constant $K>0$ such that for every $ (t,x) \in \text{supp}(\Delta_X \psi )$ we have \begin{eqnarray*}
			|\Delta_X \psi (t,x)| &\leq & |\Phi(2t/T)|  \sum_{k=1}^m \left( |\Phi^{''}(P_x)| \left( \sum_{j=1}^n |a_{k,j}| \left|\frac{2 x_j}{T}\right| \right)^2 \right. \\
			& +& \left. |\Phi'(P_x)| \left( \sum_{j = 1}^n a_{k,j}^2 \frac{2 }{T}  + \sum_{j=1}^n \frac{2 |x_j|}{T}  \left( \sum_{i=1}^n |a_{k,i}| \left|\frac{\partial a_{k,j}}{\partial x_i} \right| \right)  \right) \right. \\
			& + &  \left. \left(\sum_{i=1}^n \left|\frac{\partial a_{k,i}}{\partial x_i} \right|\right) \left(\sum_{j=1}^n |a_{k,j}| |\Phi'(P_x)| \left|\frac{2 x_j}{T} \right| \right) \right) \\
			& \leq& \frac{K |\Phi(2t/T) \left(\Phi'(P_x) + \Phi''(P_x)\right)|}{\sqrt{T}}.
		\end{eqnarray*}
		With change of variables $t \leftrightarrow \frac{2t}{T}$ and  $x_i \leftrightarrow \frac{x_i}{\sqrt{T}}$ for every $1 \leq i \leq n$, we have  \begin{eqnarray*}
			\int_{0}^T \int\limits_{\R^n} |\psi|^{ \frac{-1}{p-1}}|\Delta_X \psi|^{\frac{p}{p-1}}
			&=&  \int\limits_{\text{supp}(\Delta_X \psi)} |\psi|^{ \frac{-1}{p-1}}|\Delta_X \psi|^{\frac{p}{p-1}}    \\
			&\leq & \int_{0}^T \int\limits_{1 \leq P_x \leq 2} \left|\Phi(P_x) \Phi\left(\frac{2t}{T}\right)\right|^{\frac{-1}{p-1}} \left(\frac{K\left|\Phi\left(\frac{t}{T}\right) \left((\Phi'+ \Phi'')(P_x)\right)\right|}{\sqrt{T}}\right)^{\frac{p}{p-1}} \\
			& =& \left(\frac{K}{\sqrt{T}}\right)^{\frac{p}{p-1}} \int_{0}^T \int\limits_{1 \leq P_x \leq 2} |\Phi(P_x) |^{\frac{-1}{p-1}} 
			|(\Phi'+ \Phi'')(P_x) |^{\frac{p}{p-1}} \left|\Phi\left(\frac{2t}{T}\right)\right| \\
			&= & \left(\frac{K}{\sqrt{T}}\right)^{\frac{p}{p-1}} T^{\frac{n}{2}} \frac{T}{2} \int_{0}^2 \int\limits_{1 \leq \norm{x} \leq 2} \left|\frac{\left((\Phi'+ \Phi'')(\norm{x})\right)^p}{\Phi(\norm{x})}\right|^{\frac{1}{p-1}} |\Phi(t)|\\ 
			&\leq&  C T^{\frac{n}{2}+1 - \frac{p}{2(p-1)}}.
		\end{eqnarray*} 
		Similarly,
		\begin{eqnarray*}
			\int_{0}^T \int\limits_{\R^n}  |\psi|^{\frac{-1}{p-1}}|\psi_t|^{\frac{p}{p-1}}
			&=&  \int_{\text{supp}(\psi_t)}  |\Phi(P_x) \Phi(2t/T)|^{\frac{-1}{p-1}} \left|\frac{2}{T}  \Phi(P_x) \Phi'(2t/T)\right|^{\frac{p}{p-1}} \\
			& = &  \left(\frac{2}{T}\right)^{{\frac{p}{p-1}}}  \int_{\text{supp}(\psi_t)}  |\Phi(P_x)|  \left| \frac{\Phi'(2t/T)^p}{\Phi(2t/T)} \right|^{\frac{1}{p-1}}\\
			& = & \left(\frac{2}{T}\right)^{{\frac{p}{p-1}}} \frac{T}{2} T^{\frac{n}{2}} \int_{1}^2 \int_{0 \leq \norm{x} \leq 2}  |\Phi(\norm{x})|   \left| \frac{\Phi'(t)^p}{\Phi(t)} \right|^{\frac{1}{p-1}} \leq C T^{\frac{n}{2}+ 1 - \frac{p}{p-1}}.
		\end{eqnarray*}
		
		As, \begin{equation*}
			\int_{0}^T \int_{\R^n} f \psi = \int_{0}^T \Phi\left( \frac{2t}{T}\right) \int_{\R^n} f \Phi(P_x)  =  \frac{T}{2} \int_{0}^2 \Phi(t') \int_{\R^n} f \Phi(P_x) \geq T \int_{\R^n} f \Phi(P_x)
		\end{equation*}
		we obtain $$T \int_{\R^n} f \Phi(P_x) \leq \int_{0}^T \int_{\R^n} f \psi \leq C T^{\frac{n}{2}+1 - \frac{p}{2(p-1)}}.$$
		By dominated convergence theorem, $ \int_{\R^n} f \Phi \to \int_{\R^n} f $ as $T \to \infty$.
		Taking $T \to \infty$ and using the fact that $\frac{n}{2} < \frac{p}{2(p-1)}$ gives $\int_{\R^n} f \leq 0$.
		This contradicts the hypothesis that $\int_{\R^n} f(x) dx > 0$.
		So, no such $u$ exists and the statement is proved.
	\end{proof}

	The following theorem provides an estimate for the upper bound on the blow-up time of local in time weak solution to (\ref{MainProblem}).

	\begin{cor}\label{BlowUpTime}
		Let $p < \frac{n}{n-1}$.
		Assume that for $0 \leq f \in L^1_{loc}(\R^n)$ there exists an $\epsilon >0$ such that  $f(x) \geq \epsilon \norm{x}^{- \lambda}$ whenever $\norm{x} \geq 1$ and $ \lambda \in (0, \frac{p}{p-1})$.
		Let $T_\epsilon$ denote a time such that a weak solution exists in $[0, T_\epsilon)$.
		For every $1 \leq k \leq m$ and $1 \leq i, j \leq n$, if  the functions $a_{k,i}$ and $\frac{\partial a_{k,j}}{\partial x_i}$ are bounded then there exists a positive constant $C$ such that $T_\epsilon \leq C \epsilon^{1/(\lambda - \frac{p}{2(p-1)})}$.
	\end{cor}
	\begin{proof}
		For the test function as used in the proof of \Cref{subcritical}, we observed that $$\int_{\R^n} f \Phi(P_x) \leq C T_\epsilon^{\frac{n}{2} - \frac{p}{2(p-1)}}.$$
		For an arbitrary $T_0 < T_\epsilon$, changing the variable $x \to x':= \frac{x}{\sqrt{T_\epsilon}}$ implies that $\norm{x} \geq 1$ if and only if $\norm{x'} \geq \frac{1}{\sqrt{T_\epsilon}}$.
		Thus,
		\begin{eqnarray*}
			C T_\epsilon^{\frac{n}{2} - \frac{p}{2(p-1)}} & \geq &\int_{\R^n} f(x) \Phi(P_x) \geq  \int_{\norm{x} \geq 1} f(x) \Phi(P_x) dx  \geq  \epsilon \int_{\norm{x} \geq 1} \norm{x}^{-\lambda} \Phi(P_x) dx\\
			& =& \epsilon T_\epsilon^{-\frac{\lambda}{2}} \int_{\norm{x} \geq 1}   \left\|\frac{x}{\sqrt{T_\epsilon}}\right\|^{-\lambda}     \Phi(P_x) dx =  \epsilon T_\epsilon^{\frac{n - \lambda}{2}} \int_{\norm{x'}  \geq  \frac{1}{\sqrt{T_\epsilon}}}  \norm{x'}^{-\lambda} \Phi(P_{x'}) dx' \\
			& \geq & \epsilon T_\epsilon^{\frac{n - \lambda}{2}} \int_{\norm{x'} \geq \frac{1}{\sqrt{T_0}}}  \norm{x'}^{-\lambda} \Phi(\| x'\|^2) dx' =  \epsilon C_2 T_\epsilon^{\frac{n - \lambda}{2}}.
		\end{eqnarray*}

		As $\frac{p}{p-1} > \lambda$, we obtain $$\epsilon \leq \frac{1}{C_1} T_\epsilon^{\frac{n}{2} - \frac{p}{2(p-1)} - \frac{n - \lambda}{2} }  = \frac{1}{C_1} T_\epsilon^{- \frac{p}{2(p-1)} + \frac{\lambda}{2}} = \frac{1}{C_1} \frac{1}{T_\epsilon^{\frac{p}{2(p-1)} - \frac{\lambda}{2}}}$$ and hence $T_\epsilon^{\frac{p}{2(p-1)} - \frac{\lambda}{2}} \leq \frac{1}{C_1 \epsilon}$ which implies $T_\epsilon \leq ({\frac{1}{C_1 \epsilon}})^{\frac{1}{\frac{p}{2(p-1)} - \frac{\lambda}{2}}} = C \epsilon^{2/(\lambda - \frac{p}{p-1})}$.
	\end{proof}

	\begin{theorem}\label{critical}
		Let $p = \frac{n}{n-1}$ and $f, u_0 \in L^1_{loc}(\R^n)$ with  $u_0 \geq 0$.
		If $\int_{\R^n} f(x) dx > 0$ and the functions $a_{k,i}$ and $\frac{\partial a_{k,j}}{\partial x_i}$ are bounded  for every $1 \leq k \leq m$ and $1 \leq i,j \leq n$, then (\ref{MainProblem}) does not admit global in time weak solution.
	\end{theorem}
	\begin{proof}
		Let us suppose on the contrary that $u$ is a global in time weak solution for (\ref{MainProblem}).
		For any $0 < R < \infty$, set $$P_x :=  \frac{ \ln\left(\frac{\norm{x}}{\sqrt{R}}\right) }{\ln(\sqrt{R})}$$ and define a test function with separated variables as $$\psi(t,x) = \Phi(P_x) \Phi(t/T),$$ where \begin{equation*}
			\Phi(z) = \begin{cases} 
				1, & z \in (- \infty ,0],\\
				\searrow, & z \in [0, 1],\\
				0, & z \in (1, \infty),
			\end{cases}
		\end{equation*}
		is a function in $C^2(\R)$  satisfying the estimates $$\int\limits_{0 \leq r \leq 1} \left| \frac{ \left(\Phi' + \Phi''\right)^p\left(r \right)}{\Phi\left( r \right)}\right|^{\frac{1}{p-1}} r^{n-2} dr < \infty \text{ and } \int_{0}^1   \left|\frac{\Phi'(t)^p}{\Phi(t)}  \right|^{\frac{1}{p-1}} dt < \infty.$$
		For example, one can take a suitable translation of the function given in the proof of \Cref{subcritical}.
		Note that $$\text{supp}(\psi) \seq \{ (t,x) \in \R^+ \times \R^n : 0 \leq t \leq T, - \infty < P_x \leq 1 \}.$$
		As $- \infty < P_x \leq 1 \iff 0< \norm{x} \leq R$, the set $\text{supp}(\psi)$ is compact.
		The fact that $$\psi_t(x,t) = \frac{\partial}{\partial t}(\Phi(P_x) \Phi(t/T)) = \frac{1}{T}  \Phi(P_x) \Phi'(t/T)$$ implies  $$\text{supp}(\psi_t) \seq \{ (t,x) \in \R^+ \times \R^n : 0 \leq t \leq T, - \infty < P_x \leq 1 \}$$ and for $(t,x) \in \R^+ \times \R^n$ 
		\begin{eqnarray*}
			\Delta_X \psi(t,x) &=& \Phi(t/T) \Delta_X  \Phi(P_x) 
			= \Phi(t/T)  \sum_{k=1}^m \left( \Phi^{''}(P_x) \left( \sum_{j=1}^n a_{k,j} \frac{\partial P_x}{\partial x_j} \right)^2 \right. \\
			& +& \left. \Phi'(P_x) \left( \sum_{i,j = 1}^n a_{k,i} a_{k,j} \frac{\partial^2 P_x }{\partial x_i \partial x_j} + \sum_{j=1}^n \frac{\partial P_x}{\partial x_j} \left( \sum_{i=1}^n a_{k,i} \frac{\partial a_{k,j}}{\partial x_i} \right)  \right) \right. \\
			& + & \left. \left(\sum_{i=1}^n \frac{\partial a_{k,i}}{\partial x_i} \right) \left(\sum_{j=1}^n a_{k,j} \Phi'(P_x) \frac{\partial P_x}{\partial x_j}  \right) \right)\\
			&=& \Phi(t/T) \sum_{k=1}^m \left( \Phi^{''}(P_x) \left( \sum_{j=1}^n a_{k,j}    \frac{x_j}{\ln{(\sqrt{R})} \norm{x}^2 } \right)^2 \right. \\ 
			&+& \left. \Phi'(P_x) \left( \sum_{i, j = 1}^n a_{k,j}a_{k,i} \left(  \frac{1}{\ln{(\sqrt{R})}} \left(  \frac{\delta_{ij}}{\norm{x}^2} - \frac{2 x_ix_j}{\norm{x}^4} \right)    \right)  \right. \right. \\
			&+& \left. \left. \sum_{j=1}^n \frac{x_j}{\ln{(\sqrt{R})} \norm{x}^2  }  \left( \sum_{i=1}^n a_{k,i} \frac{\partial a_{k,j}}{\partial x_i} \right)  \right) \right. \\
			& + & \left. \left(\sum_{i=1}^n \frac{\partial a_{k,i}}{\partial x_i} \right) \left(\sum_{j=1}^n a_{k,j} \Phi'(P_x) \frac{x_j}{\ln{(\sqrt{R})} \norm{x}^2  } \right) \right)
		\end{eqnarray*}
		implies that $$\text{supp}(\Delta_X(\psi)) \seq \{ (t,x) \in \R^+ \times \R^n : 0 \leq P_x \leq 1, 0 \leq t \leq T \}.$$
		For large $R$  we have $\ln(\sqrt{R}) \geq 0$.
		Thus,  $0 \leq \frac{ \ln\left(\frac{\norm{x}}{\sqrt{R}}\right) }{\ln(\sqrt{R})} \leq 1 \iff   0 \leq  \ln\left(\frac{\norm{x}}{\sqrt{R}}\right) \leq \ln(\sqrt{R}) \iff \exp{(0)} \leq  \exp{\left(\ln\left(\frac{\norm{x}}{\sqrt{R}}\right)\right)} \leq \exp{(\ln(\sqrt{R}))}        \iff 1 \leq \frac{\norm{x}}{\sqrt{R}} \leq \sqrt{R} \iff \sqrt{R} \leq \norm{x} \leq R$.
		Hence, $$\text{supp}(\Delta_X(\psi)) \seq \{ ( t,x) \in \R^+ \times \R^n : \sqrt{R} \leq \norm{x} \leq R, 0 \leq t \leq T \}.$$

		As in \Cref{subcritical}, $$\int_{0}^T \int_{\R^n} f \psi \leq \frac{1}{p' (\frac{p}{2})^{p'-1}} \Big( \int_{0}^T \int_{\R^n} |\psi|^{\frac{-1}{p-1}}|\Delta_X \psi|^{\frac{p}{p-1}} + |\psi|^{\frac{-1}{p-1}}|\psi_t|^{\frac{p}{p-1}} \Big).$$

		Using the assumptions on $a_{k,i}$ and the fact that $\Phi$ is twice continuously differentiable, one can choose \[M = \max  \left\{  \norm{a_{k,i}(x)}_{\sup}, \norm{a_{k,i}^2(x)}_{\sup}, \left\| \frac{\partial a_{k,j}(x)}{\partial x_i}\right\|_{\sup} \right\}_{1 \leq k \leq m; 1 \leq i,j \leq n}.\]
		Note that $|x_i| \leq \norm{x}$ and $\sqrt{R} \leq |x_i| \leq R$, in the support of $\Delta_X \psi$, for every $1 \leq i \leq n$.
		Thus, there exists a constant $C>0$ such that for every $ (t,x) \in \text{supp}(\Delta_X \psi)$ we have
		\begin{eqnarray*}
			|\Delta_X \psi (t,x)|
			&\leq & C |\Phi(t/T)| \sum_{k=1}^m \left( |\Phi^{''}(P_x)| \left(  \frac{\norm{x}}{\ln{(\sqrt{R})} \norm{x}^2 } \right)^2 \right. \\ 
			&+& \left. |\Phi'(P_x)| \left( \sum_{i, j = 1}^n  \frac{|a_{k,j}a_{k,i}|}{\ln{(\sqrt{R})}} \left(   \left(  \frac{\delta_{ij}}{\norm{x}^2} \right) + \left( \frac{2 \norm{x}^2}{\norm{x}^4} \right)    \right)  \right. \right. \\
			&+& \left. \left. \sum_{j=1}^n \frac{\norm{x}}{\ln{(\sqrt{R})} \norm{x}^2  }  \left( \sum_{i=1}^n \left|a_{k,i} \frac{\partial a_{k,j}}{\partial x_i} \right| \right)  \right) \right. \\
			& + & \left. \left(\sum_{i=1}^n \left| \frac{\partial a_{k,i}}{\partial x_i} \right| \right) \left(\sum_{j=1}^n |a_{k,j}| |\Phi'(P_x)| \frac{\norm{x}}{\ln{(\sqrt{R})} \norm{x}^2  } \right) \right)\\
			&\leq & \frac{C |\Phi(t/T) \left(\Phi'(P_x) + \Phi''(P_x)\right)|}{\ln{\sqrt{R}} \norm{x}} \leq  \frac{C |\Phi(t/T) \left(\Phi'(P_x) + \Phi''(P_x)\right) |}{\ln{\sqrt{R}} \sqrt{R}}. 
		\end{eqnarray*}
		Here we have used the fact that $1$-norm is equivalent to $2$-norm on $\R^n$.

		Now, with change of the variables $x \leftrightarrow \frac{x}{\sqrt{R}}$ and $t \leftrightarrow \frac{t}{T}$ we obtain  \begin{eqnarray*}
			\int_{0}^T \int\limits_{\R^n} |\psi|^{\frac{-1}{p-1}}|\Delta_X \psi|^{\frac{p}{p-1}} &=& \int\limits_{\text{supp}(\Delta_X(\psi))} |\psi|^{\frac{-1}{p-1}}|\Delta_X \psi|^{\frac{p}{p-1}}\\
			& \leq  & \int\limits_{\text{supp}(\Delta_X(\psi))} \left|\Phi(P_x) \Phi\left(\frac{t}{T}\right)\right|^{\frac{-1}{p-1}} \left(\frac{C \left|\Phi\left(\frac{t}{T}\right) \left(\Phi'(P_x) + \Phi''(P_x)\right)\right|}{\ln{\sqrt{R}} \sqrt{R}}\right)^{\frac{p}{p-1}} \\
			& \leq & \left(\frac{C }{\ln{\sqrt{R}} \sqrt{R}} \right)^{\frac{p}{p-1}} \int\limits_{\text{supp}(\Delta_X(\psi))} \frac{|\Phi(t/T)|}{|\Phi(P_x)|^{\frac{1}{p-1}}}  \left|(\Phi' + \Phi'')(P_x)\right|^{\frac{p}{p-1}} \\
			&= & C \left(\frac{1 }{\ln{\sqrt{R}} } \right)^{\frac{p}{p-1}} T \int_{0}^1 |\Phi(t)| \int\limits_{1 \leq \norm{x} \leq \sqrt{R}} \left| \frac{ \left(\Phi' + \Phi''\right)^p\left(\frac{ \ln\left(\norm{x}\right) }{\ln(\sqrt{R})}\right)}{\Phi\left(\frac{ \ln\left(\norm{x}\right) }{\ln(\sqrt{R})}\right)}  \right|^{\frac{1}{p-1}}  \\
			&\leq & C T \ln{(\sqrt{R})^{- n-1}}.
		\end{eqnarray*}
		Where, we have used the fact that $\frac{p}{p-1} = n$.
		Similarly,
		\begin{eqnarray*}
			\int_{0}^T \int\limits_{\R^n}  |\psi|^{\frac{-1}{p-1}}|\psi_t|^{\frac{p}{p-1}} &\leq&  \int_{\text{supp}(\psi_t)}  |\Phi(P_x) \Phi(t/T)|^{\frac{-1}{p-1}} \left|\frac{1}{T}  \Phi(P_x) \Phi'(t/T) \right|^{\frac{p}{p-1}} \\
			& \leq & \left(\frac{1}{T}\right)^{\frac{p}{p-1}} T R^{\frac{n}{2}} \int_{0}^1 \int\limits_{1 \leq \norm{x} \leq \sqrt{R}}  \left|\Phi\left(\frac{ \ln\left(\norm{x}\right) }{\ln(\sqrt{R})}\right)\right|  \left|\frac{\Phi'(t)^p}{\Phi(t)}  \right|^{\frac{1}{p-1}}\\
			& = & C T^{-n + 1} R^{\frac{n}{2}} \int\limits_{1 \leq \norm{x} \leq \sqrt{R}}  \left|\Phi\left(\frac{ \ln\left(\norm{x}\right) }{\ln(\sqrt{R})}\right)\right| dx \\
			& \leq & C T^{-n + 1} R^{\frac{n}{2}} \frac{1}{\ln(\sqrt{R})} \int\limits_{0 \leq r \leq 1}  \left|\Phi\left(r\right)\right| r^{n-2} dr \\
			& \leq & C T^{-n + 1} \frac{R^{\frac{n}{2}}}{\ln(\sqrt{R})}.
		\end{eqnarray*}
		Moreover, \begin{eqnarray*}
			\int_{0}^T \int_{\R^n} f \psi &=& \int_{0}^T \Phi(t/T) \int_{\R^n} f \Phi(P_x)  =  T \int_{0}^1 \Phi(t') \int_{\R^n} f \Phi(P_x) \geq  C_1 T \int_{\R^n} f \Phi(P_x)
		\end{eqnarray*}
		implies $$\int_{\R^n} f \Phi(P_x) \leq C \left(  \ln{(\sqrt{R})^{- n-1}} +  T^{-n} \frac{R^{\frac{n}{2}}}{\ln(\sqrt{R})} \right).$$
		Substituting $T = R$ and then taking $R \to \infty$ we obtain $\int_{\R^n} f(x) dx \leq 0$ which  contradicts the hypothesis that $\int_{\R^n} f(x) dx > 0$.
		So, no such $u$ exists and the statement is proved.
	\end{proof}

	This demonstrates that $\frac{n}{n-1}$ serves as a lower bound for the possible critical exponent.
	As indicated in \Cref{SubcriticalWithZeroDivergence}, if we apply the same approach as in the classical case, i.e. assuming that all the functions $a_{k,i}$ are constant, then we obtain  $\frac{n}{n-2}$ as a lower bound for the critical exponent, a well known result (see \cite[Theorem 2.1 (a)]{Bandle}).
	Therefore, we anticipate $\frac{n}{n-1}$ to be the critical exponent for certain classes $X$ of vector fields -see \Cref{subcritical54}.
	
	\begin{theorem}\label{SubcriticalWithZeroDivergence}
		Assume that for every $1 \leq k \leq m$ and $1 \leq i \leq n$ the function $a_{k,i}$ is constant and let  $f, u_0 \in L^1_{loc}(\R^n)$ with  $u_0 \geq 0$.
		If $p < \frac{n}{n-2}$ and $\int_{\R^n} f(x) dx > 0$, then (\ref{MainProblem}) does not admit global in time weak solution.
	\end{theorem}
	\begin{proof}
		The proof is similar to that of \Cref{subcritical} apart from some necessary modifications in the computations which we describe below.
		Set $P_x$, $\psi$ and $\Phi$ as in \Cref{subcritical}.

		Then
		\begin{eqnarray*}
			\Delta_X \psi(t,x) &=& \Phi(2t/T) \Delta_X  \Phi(P_x)  \\
			&=& \Phi(2t/T)  \sum_{k=1}^m \left( \Phi^{''}(P_x) \left( \sum_{j=1}^n a_{k,j} \frac{2 x_j}{T} \right)^2 + \Phi'(P_x) \left( \sum_{j = 1}^n a_{k,j}^2 \frac{2 }{T} \right) \right)
		\end{eqnarray*} gives $$\text{supp}(\Delta_X(\psi)) \seq \{ (t,x) : 1 \leq P_x \leq 2, 0 \leq 2t/T \leq 2 \}.$$

		To prove the theorem by contradiction, 	let us suppose that $u$ is a global in time weak solution.
		As in \Cref{subcritical}, we obtain  $$\int_{0}^T \int_{\R^n} f \psi \leq \frac{1}{p' (\frac{p}{2})^{p'-1}} \Big( \int_{0}^T \int_{\R^n} |\psi|^{\frac{-1}{p-1}}|\Delta_X \psi|^{\frac{p}{p-1}} + |\psi|^{\frac{-1}{p-1}}|\psi_t|^{\frac{p}{p-1}} \Big).$$

		On the support of $\Delta_X \psi$, we have $|x_i| \leq \sqrt{T}$ for every $1 \leq i \leq n$.
		Using the assumptions on $a_{k,i}$ and the fact that $\Phi$ is twice continuously differentiable, choose $$M = \sup_{(t,x) \in (0,\infty) \times \R^n} \left\{   |a_{k,i}(x)|, |a_{k,i}(x)|^2 \right\}_{1 \leq k \leq m; 1 \leq i,j \leq n}.$$
		Thus, for large $T$ there exists a constant $K>0$ such that \begin{eqnarray*}
			|\Delta_X \psi| &\leq & |\Phi(2t/T)|  \sum_{k=1}^m \left( |\Phi^{''}(P_x)| \left( \sum_{j=1}^n |a_{k,j}| \left|\frac{2 x_j}{T}\right| \right)^2 + |\Phi'(P_x)| \left( \sum_{j = 1}^n a_{k,j}^2 \frac{2 }{T} \right) \right)\\
			&\leq& \frac{K |\Phi(2t/T) \left(\Phi'(P_x) + \Phi''(P_x)\right)|}{T}.
		\end{eqnarray*}

		Now, \begin{eqnarray*}
			\int_{0}^T \int\limits_{\R^n} |\psi|^{\frac{-1}{p-1}}|\Delta_X \psi|^{\frac{p}{p-1}}
			& \leq  & \int\limits_{\text{supp}(\Delta_X(\psi))} \left|\Phi(P_x) \Phi\left(\frac{2t}{T}\right)\right|^{\frac{-1}{p-1}} \left(\frac{K \left|\Phi\left(\frac{2t}{T}\right)(\Phi'+ \Phi'')(P_x) \right|}{T}\right)^{\frac{p}{p-1}} \\
			&\leq & \left(\frac{K }{T}\right)^{\frac{p}{p-1}} T^{\frac{n}{2}} \frac{T}{2} \int_{0}^2 \int_{1 \leq \norm{x} \leq 2} \left| \frac{  ( \Phi' + \Phi'' )^p (\norm{x}) }{\Phi(\norm{x})}   \right|^{\frac{1}{p-1}} \Phi(t)\\
			& \leq & C T^{\frac{n}{2}+1 - \frac{p}{p-1}}.
		\end{eqnarray*} 
		
		Similarly,
		\begin{eqnarray*}
			\int_{0}^T \int\limits_{\R^n}   | \psi|^{\frac{-1}{p-1}} |\psi_t|^{\frac{p}{p-1}} 
			&\leq& \int_{\text{supp}(\psi_t)}  |\Phi(P_x) \Phi(2t/T)|^{\frac{-1}{p-1}} \left|\frac{2}{T}  \Phi(P_x) \Phi'(2t/T)\right|^{\frac{p}{p-1}} \\
			& \leq &  \left(\frac{2}{T}\right)^{{\frac{p}{p-1}}} \int_{\text{supp}(\psi_t)}  |\Phi(P_x)|  |\Phi(2t/T)|^{\frac{-1}{p-1}}|  \Phi'(2t/T)|^{\frac{p}{p-1}}\\
			& \leq & \left(\frac{2}{T}\right)^{{\frac{p}{p-1}}} \frac{T}{2} T^{\frac{n}{2}} \int_{0}^2 \int_{0 \leq \norm{x} \leq 2}  |\Phi(\norm{x})|  |\Phi(t)|^{\frac{-1}{p-1}}|  \Phi'(t)|^{\frac{p}{p-1}}.
		\end{eqnarray*}
		
		Moreover, \begin{eqnarray*}
			\int_{0}^T \int_{\R^n} f \psi &=& \int_{0}^T \Phi(2t/T) \int_{\R^n} f \Phi(P_x)    \geq T \int_{\R^n} f \Phi(P_x).
		\end{eqnarray*}
		
		Consequently, we obtain $$T \int_{\R^n} f \Phi(P_x) \leq \int_{0}^T \int_{\R^n} f \psi \leq  C ( T^{\frac{n}{2}+1 - \frac{p}{p-1}} + T^{\frac{n}{2}+1 - \frac{p}{p-1}}).$$
		Using the fact that $\frac{n}{2} < \frac{p}{p-1}$, taking $T \to \infty$ and applying the Lebesgue dominated convergence theorem we obtain $\int_{\R^n} f \Phi \to \int_{\R^n} f \leq 0$.
		This contradicts the hypothesis that $\int_{\R^n} f > 0$.
		So, no such $u$ exists and the statement is proved.
	\end{proof}

	\begin{remark}\label{remark}
		With a minor modification in the test functions used to obtain results of this section, the non-negativity assumption on $u_0$ may be removed.
		The only use of non-negativity of $u_0$ in the proofs was to get rid of the term $\int_{\R^n} u_0 \psi(0, x)$ in equation 3.1 so that we have inequality 3.2.
		However, this can be achieved by considering a test function $\psi$ which vanishes at time $t = 0$.
		In particular, in \Cref{subcritical}, instead of the test function $\psi(t,x) = \Phi(P_x) \Phi(2t/T)$, one may use $\psi(t,x) = \Phi(P_x) \Phi_1(t/T)$ where $\Phi_1 \in C^2_0(\R^n)$ is an appropriate function supported in $(0,1)$.
		It will be demonstrated through the following section that similar steps then lead to the desired conclusion.
	\end{remark}

	\section{On certain well-studied vector fields}\label{Section4}
	
	We now delve into the study of more general vector fields and undertake the delicate task of constructing test functions to determine conditions under which (\ref{MainProblem}) does not admit a global in time weak solution.
	As mentioned in \cite{Biagi}, what is special about some of these vector fields is that although they fall within the class of vector fields being considered in this paper (with $\dive X_k = 0$), they do not give rise to a stratified Lie group structure on the space (in another article, we find the critical Fujita exponent on stratified Lie groups).
	Therefore, our findings generalize the subcritical case examined in \cite{Torebek}, where the group structure on the domain significantly influenced the proofs.

	\begin{theorem}\label{subcritical3}
		Let $p < \frac{2^n-1}{2^n-3}$, $X_1 = \frac{\partial}{\partial x_1}$ and $X_2 = x_1 \frac{\partial}{\partial x_2} + x_1^2 \frac{\partial}{\partial x_3} + \dots x_1^{n-1} \frac{\partial}{\partial x_n}$ on $\R^n$.
		If $f \in L^1_{loc}(\R^n)$ with $\int_{\R^n} f(x) dx > 0$, then (\ref{MainProblem}) does not admit  a global in time weak solution.
	\end{theorem}
	\begin{proof}
		Let us suppose on the contrary that $u$ is a global in time weak solution.
		As in \Cref{subcritical1}, for a test function $\psi$ which is zero at time $t = 0$, we obtain	\begin{eqnarray}\label{Inequality}
			\int_{0}^T \int_{\R^n} f \psi + \int_{0}^T \int_{\R^n} |u|^p \psi & = & - \int_{0}^T \int_{\R^n} u \psi_t   - \int_{0}^T \int_{\R^n} u \Delta_X \psi \nonumber  \\
			& \leq &  \int_{0}^T \int_{\R^n} | u || \psi_t| + |u|| \Delta_X \psi|.
		\end{eqnarray}
		Let $$d(x) = x_1^{2^{n}} + x_2^{2^{n-1}} + x_3^{2^{n-2}} + \dots + x_n^{2} \text{ and } f(a) = \frac{a}{T^{2^{n-1}}}, \text{ for every } a \in \R.$$
		Define a test function with separated variables as $$\psi_T(t,x) = \Phi_1(t/T) \Phi_2((f \circ d)(x)),$$ where $0 \lneq \Phi_1 \in C^\infty (\R)$ with $\text{supp}(\Phi_1) \seq (0,1)$ satisfies the estimate $$\int_0^1  |\Phi_1(t)|^{\frac{-1}{p-1}}| \Phi_1'(t)|^{\frac{p}{p-1}}  dt < \infty$$ and \begin{equation*}
			\Phi_2(z) = \begin{cases} 
				1, & z \in [0,1],\\
				\searrow, & z \in [1, 2],\\
				0, & z \in (2, \infty),
			\end{cases}
		\end{equation*}
		is a function in $C^2_0(\R^+)$ satisfying the estimate $$\int_{0 \leq P'(x) \leq 1} \left| \frac{(\Phi_2'+ \Phi_2'')^p(P'(x))}{\Phi_2(P'(x))}   \right|^{\frac{1}{p-1}} < \infty,$$  where  $P'(x) := (x_1')^{2^n} + (x_2')^{2^{n-1}} + (x_3')^{2^{n-2}} + \dots + (x_n')^{2}$ with $x_i':= \frac{x_i}{\sqrt{T}^{2^{i-1}}}$ for every $1 \leq i \leq n$.
		Note that $$\text{supp}(\psi_T) \seq \left\{ (t,x) \in \R^+ \times \R^n : 0 \leq \frac{t}{T} \leq 1, 0 \leq P(x) \leq 2 \right\}.$$

		Now $$\Delta_X \psi_T(t,x) = \Phi_1(t/T) \Delta_X ((\Phi_2 \circ f) \circ d)(x)$$ and for $\Psi := \Phi_2 \circ f$ we obtain 
		\begin{eqnarray*}
			\Delta_X ((\Phi_2 \circ f) \circ d) &=& \Delta_X (\Psi \circ d) =  \sum_{k=1}^2 \left(  X_k^2 (\Psi \circ d)  \right)\\
			&=& \sum_{k=1}^2 \left( \Psi''(d)  (X_k(d))^2 + \Psi'(d) X_k^2(d)  \right)
		\end{eqnarray*}implying that $$\text{supp}(\Delta_X \psi_T) = \left\{ (t,x) \in \R^+ \times \R^n : 0 \leq \frac{t}{T} \leq 1 \leq (f \circ d )(x) \leq 2 \right\}.$$
		Note that $$ \Psi''(d) = (\Phi_2 \circ f)''(d) = \left(\frac{\Phi_2'(f)(d)}{T^{2^{n-1}}}\right)' = \frac{\Phi_2''(f)(d)}{T^{2^{n}}}.$$
		Therefore, for $x \in \text{supp}(\Delta_X \psi_T)$ we have $x_i \leq C \sqrt{T}^{2^{i-1}}$ for every $1 \leq i \leq n$.
		Since $$X_1(d(x)) = 2^n x_1^{2^n-1} ~ \text{ , } ~ X_2(d(x)) =  x_1 2^{n-1} x_2^{2^{n-1}-1} + 2^{n-2} x_1^2 x_3^{2^{n-2}-1} + \dots + 2 x_1^{n-1} x_n,$$ $$X_1^2 (d(x)) = 2^n (2^n-1) x_1^{2^n-2}$$ and $$X_2^2 (d(x)) = x_1^2 2^{n-1} (2^{n-1}-1) x_2^{2^{n-1}-2} + 2^{n-2} (2^{n-2}-1) x_1^4 x_3^{2^{n-2}-2} + \dots + 2 x_1^{2n-2},$$ we have
		\begin{eqnarray*}
			\Delta_X ((\Phi_2 \circ f) \circ d) &=&  \sum_{k=1}^2 \left( \Psi''(d)  (X_k(d))^2 + \Psi'(d) X_k^2(d)  \right)\\
			&=& \frac{\Phi_2''(f)(d)}{T^{2^{n+1}}} \left( (2^n x_1^{2^n-1})^2 + (x_1 2^{n-1} x_2^{2^{n-1}-1}  + \dots + 2 x_1^{n-1} x_n)^2 \right) \\ 
			&+& \left(\frac{\Phi_2'(f)(d)}{T^{2^n}}\right) \left( 2^n (2^n-1) x_1^{2^n-2} + x_1^2 2^{n-1} (2^{n-1}-1) x_2^{2^{n-1}-2} + \right. \\ 
			&& \left. 2^{n-2} (2^{n-2}-1) x_1^4 x_3^{2^{n-2}-2} + \dots + 2 x_1^{2n-2} \right)
		\end{eqnarray*}
		and hence $$\left|\Delta_X ((\Phi_2 \circ f) \circ d)  \right|\leq \frac{C |(\Phi_2' + \Phi_2'')(f)(d)|}{T}.$$
		Also, $$(\psi_T)_t(t,x) = \frac{\partial}{\partial t}\left(\Phi_1\left(\frac{t}{T}\right) \Phi_2\left((f \circ d )(x)\right)\right) = \frac{\Phi_1'(t/T)\Phi_2\left((f \circ d )(x)\right) }{T}$$ implies that $$\text{supp}((\psi_T)_t) \seq \left\{ (t,x) \in \R^+ \times \R^n : 0 \leq \frac{t}{T} \leq 1, 0 \leq (f \circ d )(x) \leq 2 \right\}$$ and  $$\left| (\psi_T)_t \right| \leq \frac{C}{T}.$$

		As in \Cref{subcritical}, one obtains
		\[\int_{0}^T \int_{\R^n} f \psi_T \leq C \Big( \int_{0}^T \int_{\R^n} |\psi_T|^{\frac{-1}{p-1}}|\Delta_X \psi_T|^{\frac{p}{p-1}} + |\psi_T|^{\frac{-1}{p-1}}|(\psi_T)_t|^{\frac{p}{p-1}} \Big).\]
		Now, for $t' = \frac{t}{T}$, $P(x):= (f \circ d)(x)$ and $S := \text{supp}(\Delta_X \psi_T)$, we have \begin{eqnarray*}
			\int_{0}^T \int\limits_{\R^n} \psi_T|^{ \frac{-1}{p-1}}|\Delta_X \psi|^{\frac{p}{p-1}} 
			&=& \int_S |\psi_T|^{ \frac{-1}{p-1}}|\Delta_X \psi|^{\frac{p}{p-1}} \\
			& \leq& \int_S |\Phi_2(P(x)) |^{\frac{-1}{p-1}} \left(\frac{C |(\Phi_2' + \Phi_2'')(f)(d)|}{T}\right)^{\frac{p}{p-1}} |\Phi_1(t/T)| \\
			&= & C\left(\frac{1}{T}\right)^{\frac{p}{p-1}} \sqrt{T}^{1+2+2^2+ \dots + 2^{n-1}} T  \int\limits_{0 \leq P'(x) \leq 2} \left| \frac{(\Phi_2' + \Phi_2'')^p(P'(x))}{\Phi_2(P'(x))} \right|^{\frac{1}{p-1}} \\
			&= & C \left(\frac{1}{T}\right)^{\frac{p}{p-1}} \sqrt{T}^{2^n-1} T.
		\end{eqnarray*} 
		Similarly, for $S_1 : = \text{supp}((\psi_T)_t)$, we have
		\begin{eqnarray*}
			\int_{0}^T \int\limits_{\R^n} |\psi_T|^{\frac{-1}{p-1}}|(\psi_T)_t|^{\frac{p}{p-1}}
			&=&  \int_{S_1} |\psi_T|^{\frac{-1}{p-1}}|(\psi_T)_t|^{\frac{p}{p-1}}\\
			&=&  \int_{S_1}  |\Phi_2(P(x)) \Phi_1(t/T)|^{\frac{-1}{p-1}} \left|\frac{1}{T}  \Phi_2(P(x)) \Phi_1'(t/T)\right|^{\frac{p}{p-1}} \\
			& = &  \left(\frac{1}{T}\right)^{{\frac{p}{p-1}}} \int_{S_1} |\Phi_2(P(x))|  |\Phi(t/T)|^{\frac{-1}{p-1}}|  \Phi_1'(t/T)|^{\frac{p}{p-1}}\\
			& = & \left(\frac{2}{T}\right)^{{\frac{p}{p-1}}} \sqrt{T}^{2^n-1} T \int_{0}^1 \int\limits_{0 \leq P'(x) \leq 2} |\Phi_2(P'(x))| |\Phi_1(t')|^{\frac{-1}{p-1}}| \Phi_1'(t')|^{\frac{p}{p-1}}.
		\end{eqnarray*}
		
		As, \begin{equation*}
			\int_{0}^T \int\limits_{\R^n} f \psi_T = \int_{0}^T \Phi_1(t/T) \int\limits_{\R^n} f \Phi_2(P(x))  =  T \int_{0}^1 \Phi_1(t') \int\limits_{\R^n} f \Phi_2(P(x)) = C T \int\limits_{\R^n} f \Phi_2(P(x))
		\end{equation*}
		we obtain $$\int_{\R^n} f \Phi_2(P(x)) \leq C T^{-1} \int_{0}^T \int_{\R^n} f \psi_T \leq C T^{ \frac{2^n-1}{2} - \frac{p}{p-1}}.$$
		By dominated convergence theorem, $ \int_{\R^n} f \Phi_2 \to \int_{\R^n} f $ as $T \to \infty$.
		Using the fact that $p < \frac{2^n-1}{2^n-3}$ and taking $T \to \infty$ we obtain  $\int_{\R^n} f \leq 0$.
		This contradicts the hypothesis that $\int_{\R^n} f > 0$.
		So, no such $u$ exists and the statement is proved.
	\end{proof}
	
	\begin{cor}\label{BlowUpTime3}
		Under the terminology and assumptions of \Cref{subcritical3}, assume further that $0 \leq f \in L^1_{loc}(\R^n)$ and  let there be an $\epsilon >0$ such that  $f(x) \geq \epsilon d(x)^{- \lambda}$ whenever $ d(x) \geq 1$ with $ \lambda \in \left(0, \frac{p}{(2^{n-1})(p-1)}\right)$.
		Let $T_\epsilon$ denote a time such that a weak solution exists in $[0, T_\epsilon)$.
		Then there exists a positive constant $C$ such that $T_\epsilon \leq C \epsilon^{1/\left((2^{n-1}) \lambda - \frac{p}{p-1}\right)}$.
	\end{cor}
	\begin{proof}
		The proof follows along the same lines of \Cref{BlowUpTime}, hence the details are skipped.
	\end{proof}

	Going forward, we will refrain from providing detailed proofs and instead focus solely on presenting the key steps of the proof. A detailed proof can be constructed following the approach outlined in the prro of \Cref{subcritical3}.

	\begin{theorem}\label{subcritical51}
		Let $p < \frac{7}{5}$, $X_1 = \frac{\partial}{\partial x_1}$ and $X_2 = x_1 x_2 \frac{\partial}{\partial x_3}$ on $\R^3$.
		If $f \in L^1_{loc}(\R^3)$ with $\int_{\R^3} f(x) dx > 0$, then (\ref{MainProblem}) does not admit global in time weak solution.
	\end{theorem}
	\begin{proof}
		Let us suppose on the contrary that $u$ is a global in time weak solution.
		For the test function $\psi$ defined below, we have
		\[\int_{0}^T \int_{\R^2} f \psi_T \leq C \Big( \int_{0}^T \int_{\R^2} |\psi_T|^{\frac{-1}{p-1}}|\Delta_X \psi_T|^{\frac{p}{p-1}} + |\psi_T|^{\frac{-1}{p-1}}|(\psi_T)_t|^{\frac{p}{p-1}} \Big).\]
		Let $$d(x) = x_1^{16} + x_2^{8} + x_3^4 \text{ and } f(a) = \frac{a}{T^{8}}, \text{ for } a \in \R.$$
		Define a test function with separated variables as $$\psi_T(t,x) = \Phi_1(t/T) \Phi_2((f \circ d)(x)),$$ where $\Phi_1$ is as earlier and $\Phi_2$ satisfies the estimate $$\int_{0 \leq P'(x) \leq 1} \left| \frac{(\Phi_2'+ \Phi_2'')^p(P'(x))}{\Phi_2(P'(x))}   \right|^{\frac{1}{p-1}}   < \infty$$  where  $P'(x) := (x_1')^{16} + (x_2')^8 + x_3'^4$ with $x_1':= \frac{x_1}{\sqrt{T}}$, $x_2':= \frac{x_2}{T}$ and $x_3' = \frac{x_3}{T^{2}}$.
		For $x \in \text{supp}(\Delta_X \psi_T)$ we have $|x_1| \leq C \sqrt{T}$ and $|x_2| \leq C T$, $|x_3|\leq C T^2$ and hence $$\left|\Delta_X ((\Phi_2 \circ f) \circ d)  \right|\leq \frac{C \left( \Phi_2'((f \circ d)(x)) + \Phi_2''((f \circ d)(x)) \right)}{T}.$$
		Similarly, 
		$$\left| (\psi_T)_t \right| \leq \frac{C}{T}.$$
		Changing the variables $x \leftrightarrow x'$ and $t \leftrightarrow t':= \frac{t}{T}$ gives \begin{eqnarray*}
			\int_{0}^T \int\limits_{\R^2} |\psi_T|^{ \frac{-1}{p-1}}|\Delta_X \psi_T|^{\frac{p}{p-1}}
			&\leq &  C\left(\frac{1}{T}\right)^{\frac{p}{p-1}} T^{\frac{9}{2}}.
		\end{eqnarray*} 
		and  
		\begin{eqnarray*}
			\int_{0}^T \int\limits_{\R^2} 	|\psi_T|^{\frac{-1}{p-1}}|(\psi_T)_t|^{\frac{p}{p-1}}
			&\leq & C \left(\frac{1}{T}\right)^{\frac{p}{p-1}} T^{\frac{9}{2}}.
		\end{eqnarray*}
		
		Since \begin{equation*}
			\int_{0}^T \int\limits_{\R^2} f \psi_T =  T \int_{0}^1 \Phi_1(t') \int\limits_{\R^2} f \Phi_2(P(x)) = C T \int\limits_{\R^2} f \Phi_2(P(x)),
		\end{equation*}
		we obtain $$\int_{\R^2} f \Phi_2(P(x)) \leq C T^{-1}\int_{0}^T \int_{\R^2} f \psi_T \leq C  T^{\frac{7}{2} - \frac{p}{p-1}} .$$
		Using the fact that $p < \frac{7}{5}$ and taking $T \to \infty$ we obtain  $\int_{\R^2} f \leq 0$.
		This contradicts the hypothesis that $\int_{\R^2} f > 0$.
		So, no such $u$ exists and the statement is proved.
	\end{proof}
	
	\begin{cor}\label{BlowUpTime51}
		Under the terminology and assumptions of \Cref{subcritical51}, assume further that $0 \leq f \in L^1_{loc}(\R^n)$ and  let there be an $\epsilon >0$ such that  $f(x) \geq \epsilon d(x)^{- \lambda}$ whenever $ d(x) \geq 1$ with $ \lambda \in \left(0, \frac{p}{8(p-1)}\right)$.
		Let $T_\epsilon$ denote a time such that a weak solution exists in $[0, T_\epsilon)$.
		Then there exists a positive constant $C$ such that $T_\epsilon \leq C \epsilon^{1/\left(8 \lambda - \frac{p}{p-1}\right)}$.
	\end{cor}
	\begin{proof}
		The proof follows along the same lines of \Cref{BlowUpTime}, hence the details are skipped.
	\end{proof}

	\begin{theorem}\label{subcritical53}
		Let $p < \frac{n^2+n}{n^2 + n - 4}$, $X_1 = \frac{\partial}{\partial x_1}$ and $X_2 = x_1  \frac{\partial}{\partial x_2} + x_2 \frac{\partial}{\partial x_3} + \dots + x_{n-1} \frac{\partial}{\partial x_n}$  on $\R^n$.
		If $f \in L^1_{loc}(\R^n)$ with $\int_{\R^n} f(x) dx > 0$, then (\ref{MainProblem}) does not admit global in time weak solution.
	\end{theorem}
	\begin{proof}
		Let us suppose on the contrary that $u$ is a global in time weak solution.
		For the test function $\psi$ defined below, we have 
		\[\int_{0}^T \int_{\R^2} f \psi_T \leq C \Big( \int_{0}^T \int_{\R^2} |\psi_T|^{\frac{-1}{p-1}}|\Delta_X \psi_T|^{\frac{p}{p-1}} + |\psi_T|^{\frac{-1}{p-1}}|(\psi_T)_t|^{\frac{p}{p-1}} \Big).\]
		Let $$d(x) = x_1^{2n!} + x_2^{\frac{2n!}{2}} + x_3^{\frac{2n!}{3}} + \dots + x_n^{\frac{2n!}{n}} \text{ and } f(a) = \frac{a}{T^{n!}}, \text{ for } a \in \R.$$
		Define a test function with separated variables as $$\psi_T(t,x) = \Phi_1(t/T) \Phi_2((f \circ d)(x)),$$ where $\Phi_1$ is as earlier and $\Phi_2$ satisfies the estimate  $$\int_{0 \leq P'(x) \leq 1} \left| \frac{(\Phi_2'+ \Phi_2'')^p(P'(x))}{\Phi_2(P'(x))}   \right|^{\frac{1}{p-1}}   < \infty$$  where  $P'(x) := x_1'^{2n!} + x_2'^{\frac{2n!}{2}} + x_3'^{\frac{2n!}{3}} + \dots + x_n'^{\frac{2n!}{n}}$ with $x_i':= \frac{x_i}{T^{i/2}}$ for $1 \leq i \leq n$.
		For $x \in \text{supp}(\Delta_X \psi_T)$ and each $1 \leq i \leq n$ we have $|x_i| \leq C T^{\frac{i}{2}}$.
		Hence $$\left|\Delta_X ((\Phi_2 \circ f) \circ d)  \right|\leq \frac{C \left( \Phi_2'((f \circ d)(x)) + \Phi_2''((f \circ d)(x)) \right)}{T}$$ and similarly $$\left| (\psi_T)_t \right| \leq \frac{C}{T}.$$

		Changing the variable $x \leftrightarrow x'$ and $t \leftrightarrow t':= \frac{t}{T}$ gives \begin{eqnarray*}
			\int_{0}^T \int\limits_{\R^2} |\psi_T|^{ \frac{-1}{p-1}}|\Delta_X \psi_T|^{\frac{p}{p-1}}
			&\leq &  C\left(\frac{1}{T}\right)^{\frac{p}{p-1}} T^{\frac{n (n+1) + 4}{4}}
		\end{eqnarray*} 
		and 
		\begin{eqnarray*}
			\int_{0}^T \int\limits_{\R^2} 	|\psi_T|^{\frac{-1}{p-1}}|(\psi_T)_t|^{\frac{p}{p-1}}
			&\leq & C \left(\frac{1}{T}\right)^{\frac{p}{p-1}} T^{\frac{n (n+1) + 4}{4}}.
		\end{eqnarray*}
		
		Since \begin{equation*}
			\int_{0}^T \int\limits_{\R^2} f \psi_T =  C T \int\limits_{\R^2} f \Phi_2(P(x)),
		\end{equation*}
		we obtain $$\int_{\R^2} f \Phi_2(P(x)) \leq C T^{-1}\int_{0}^T \int_{\R^2} f \psi_T \leq C  T^{\frac{n (n+1) }{4} - \frac{p}{p-1}} .$$
		The assumption that $p < \frac{n^2+n}{n^2 + n - 4}$	now contradicts the hypothesis that $\int_{\R^2} f > 0$.
		So, no such $u$ exists and the statement is proved.
	\end{proof}
	
	Note that since $n^2 + n$ is always even, $\frac{n^2+n}{n^2 + n - 4}$ is still of the form $\frac{\alpha}{\alpha-2}$ for some $\alpha$.

	\begin{cor}\label{BlowUpTime53}
		Under the terminology and assumptions of \Cref{subcritical53}, assume further that $0 \leq f \in L^1_{loc}(\R^n)$ and  let there be an $\epsilon >0$ such that  $f(x) \geq \epsilon d(x)^{- \lambda}$ whenever $ d(x) \geq 1$ with $ \lambda \in \left(0, \frac{p}{n!(p-1)}\right)$.
		Let $T_\epsilon$ denote a time such that a weak solution exists in $[0, T_\epsilon)$.
		Then there exists a positive constant $C$ such that $T_\epsilon \leq C \epsilon^{1/\left(n! \lambda - \frac{p}{p-1}\right)}$.
	\end{cor}
	
	\begin{proof}
		The proof follows along the same lines of \Cref{BlowUpTime}, hence the details are skipped.
	\end{proof}

	\begin{theorem}\label{subcritical54}
		Let $p < \frac{3}{2}$, $X_1 = \frac{\partial}{\partial x_1}$, $X_2 =  \frac{\partial}{\partial x_2}$ and $X_3 =  (x_1^2 + x_2^2 + x_1 x_2) \frac{\partial}{\partial x_3}$ on $\R^3$.
		If $f \in L^1_{loc}(\R^3)$ with $\int_{\R^3} f(x) dx > 0$, then (\ref{MainProblem}) does not admit global in time weak solution.
	\end{theorem}
	\begin{proof}
		Let us suppose on the contrary that $u$ is a global in time weak solution.
		For the test function $\psi$ defined below, we have
		\[\int_{0}^T \int_{\R^2} f \psi_T \leq C \Big( \int_{0}^T \int_{\R^2} |\psi_T|^{\frac{-1}{p-1}}|\Delta_X \psi_T|^{\frac{p}{p-1}} + |\psi_T|^{\frac{-1}{p-1}}|(\psi_T)_t|^{\frac{p}{p-1}} \Big).\]
		Let $$d(x) = x_1^{8} + x_2^{8} + x_3^2 \text{ and } f(a) = \frac{a}{T^{4}}, \text{ for } a \in \R.$$
		Define a test function with separated variables as $$\psi_T(t,x) = \Phi_1(t/T) \Phi_2((f \circ d)(x)),$$ where $\Phi_1$ is as earlier and $\Phi_2$ satisfies the estimate  $$\int_{0 \leq P'(x) \leq 1} \left| \frac{(\Phi_2'+ \Phi_2'')^p(P'(x))}{\Phi_2(P'(x))}   \right|^{\frac{1}{p-1}}   < \infty$$  where  $P'(x) := x_1'^{8} + x_2'^{8} + x_3'^2$ with $x_1':= \frac{x_1}{\sqrt{T}}$, $x_2':= \frac{x_2}{\sqrt{T}}$ and $x_3':= \frac{x_3}{T^2}$.
		
		For $x \in \text{supp}(\Delta_X \psi_T)$ we have $|x_1| \leq C \sqrt{T}$, $|x_2| \leq C \sqrt{T}$ and $|x_3| \leq C T^2$.
		Hence $$\left|\Delta_X ((\Phi_2 \circ f) \circ d)  \right|\leq \frac{C \left( \Phi_2'((f \circ d)(x)) + \Phi_2''((f \circ d)(x)) \right)}{T}$$ and similarly $$\left| (\psi_T)_t \right| \leq \frac{C}{T}.$$
		Changing the variable $x \leftrightarrow x'$ and $t \leftrightarrow t':= \frac{t}{T}$ gives \begin{eqnarray*}
			\int_{0}^T \int\limits_{\R^2} |\psi_T|^{ \frac{-1}{p-1}}|\Delta_X \psi_T|^{\frac{p}{p-1}}
			&\leq &  C\left(\frac{1}{T}\right)^{\frac{p}{p-1}} T^4
		\end{eqnarray*} 
		and
		\begin{eqnarray*}
			\int_{0}^T \int\limits_{\R^2} 	|\psi_T|^{\frac{-1}{p-1}}|(\psi_T)_t|^{\frac{p}{p-1}}
			&\leq & C \left(\frac{1}{T}\right)^{\frac{p}{p-1}} T^4.
		\end{eqnarray*}
		
		Since \begin{equation*}
			\int_{0}^T \int\limits_{\R^2} f \psi_T   =  T \int_{0}^1 \Phi_1(t') \int\limits_{\R^2} f \Phi_2(P(x)) = C T \int\limits_{\R^2} f \Phi_2(P(x)),
		\end{equation*}
		we obtain $$\int_{\R^2} f \Phi_2(P(x)) \leq C T^{-1}\int_{0}^T \int_{\R^2} f \psi_T \leq C  T^{3 - \frac{p}{p-1}} .$$
		Using the fact that $p < \frac{3}{2}$ and taking $T \to \infty$ we obtain  $\int_{\R^2} f \leq 0$.
		This contradicts the hypothesis that $\int_{\R^2} f > 0$.
		So, no such $u$ exists and the statement is proved.
	\end{proof}
	
	\begin{cor}\label{BlowUpTime54}
		Under the terminology and assumptions of \Cref{subcritical54}, assume further that $0 \leq f \in L^1_{loc}(\R^n)$ and  let there be an $\epsilon >0$ such that  $f(x) \geq \epsilon d(x)^{-\lambda}$ whenever $ d(x) \geq 1$ with $\lambda \in \left(0, \frac{p}{4(p-1)}\right)$.
		Let $T_\epsilon$ denote a time such that a weak solution exists in $[0, T_\epsilon)$.
		Then there exists a positive constant $C$ such that $T_\epsilon \leq C \epsilon^{1/\left(4 \lambda - \frac{p}{p-1}\right)}$.
	\end{cor}
	\begin{proof}
		The proof follows along the same lines of \Cref{BlowUpTime}, hence the details are skipped.
	\end{proof}

	\begin{theorem}\label{subcritical55}
		Let $X_1 = \frac{\partial}{\partial x_1}$ and $X_j = x_1^{k_{j,1}} x_2^{k_{j,2}} \dots x_{j-1}^{k_{j,j-1}} \frac{\partial}{\partial x_j}$ on $\R^n$ and $$p < \frac{n + \sum_{j=2}^n \sum_{i =1 }^{j-1} \left( \sum_{l_1 = i}^{j-1} \sum_{l_2 = i-1}^{l_1-1} \sum_{l_3 = i-2}^{l_2-1} \dots \sum_{l_i = 1}^{l_{i-1} - 1} k_{j, l_1} k_{l_1, l_2} k_{l_2, l_3} \dots k_{l_i, l_{i-1}} \right)}{n - 2 + \sum_{j=2}^n \sum_{i =1 }^{j-1} \left( \sum_{l_1 = i}^{j-1} \sum_{l_2 = i-1}^{l_1-1} \sum_{l_3 = i-2}^{l_2-1} \dots \sum_{l_i = 1}^{l_{i-1} - 1} k_{j, l_1} k_{l_1, l_2} k_{l_2, l_3} \dots k_{l_i, l_{i-1}} \right) }.$$
		If $f \in L^1_{loc}(\R^n)$ with $\int_{\R^n} f(x) dx > 0$, then (\ref{MainProblem}) does not admit global in time weak solution.
	\end{theorem}
	\begin{proof}
		Let us suppose on the contrary that $u$ is a global in time weak solution.
		For the test function $\psi$ defined below, we have
		\[\int_{0}^T \int_{\R^2} f \psi_T \leq C \Big( \int_{0}^T \int_{\R^2} |\psi_T|^{\frac{-1}{p-1}}|\Delta_X \psi_T|^{\frac{p}{p-1}} + |\psi_T|^{\frac{-1}{p-1}}|(\psi_T)_t|^{\frac{p}{p-1}} \Big).\]
		Let $$d(x) = x_1^{2 r_1} + x_2^{2 r_2} + \dots x_n^{2 r_n} \text{ and } f(a) = \frac{a}{T^{r_1}}, \text{ for } a \in \R,$$ where, for every $2 \leq j \leq n$ we have  $$r_j = \frac{r_1}{1 + R_{j,1} + R_{j,2} + \dots R_{j, (j-1)}}$$ with, for every $1 \leq i \leq n$ we define $$R_{j,s}:= \sum_{l_1 = i}^{j-1} \sum_{l_2 = i-1}^{l_1-1} \sum_{l_3 = i-2}^{l_2-1} \dots \sum_{l_i = 1}^{l_{s-1}-1} k_{j, l_1} k_{l_1, l_2} k_{l_2, l_3} \dots k_{l_i, l_{i-1}} .$$
		
		Define a test function with separated variables as $$\psi_T(t,x) = \Phi_1(t/T) \Phi_2((f \circ d)(x)),$$ where $\Phi_1$ is as earlier and $\Phi_2$ satisfies the estimate  $$\int_{0 \leq P'(x) \leq 1} \left| \frac{(\Phi_2'+ \Phi_2'')^p(P'(x))}{\Phi_2(P'(x))}   \right|^{\frac{1}{p-1}}   < \infty$$  where   $P'(x) := x_1'^{2 r_1} + x_2'^{2 r_2} + \dots x_n'^{2 r_n}$ with $x_j':= \frac{x_1}{T^{\frac{r_1}{2 r_j}}}$ for every $1 \leq j \leq n$.
		For $x \in \text{supp}(\Delta_X \psi_T)$ and every $1 \leq j \leq n$ we have $|x_j| \leq C T^{\frac{r_1}{2 r_j}}$ and hence $$\left|\Delta_X ((\Phi_2 \circ f) \circ d)  \right|\leq \frac{C \left( \Phi_2'((f \circ d)(x)) + \Phi_2''((f \circ d)(x)) \right)}{T}.$$
		Similarly, $$\left| (\psi_T)_t \right| \leq \frac{C}{T}.$$
		
		Changing the variables $x \leftrightarrow x'$ and $t \leftrightarrow t':= \frac{t}{T}$ gives \begin{eqnarray*}
			\int_{0}^T \int\limits_{\R^2} |\psi_T|^{ \frac{-1}{p-1}}|\Delta_X \psi_T|^{\frac{p}{p-1}}
			&\leq &  C\left(\frac{1}{T}\right)^{\frac{p}{p-1}} T^{\frac{n + 2 + \sum_{j=2}^n \sum_{i =1 }^{j-1} R_{j, i} }{2}}.
		\end{eqnarray*} 
		and 
		\begin{eqnarray*}
			\int_{0}^T \int\limits_{\R^2} 	|\psi_T|^{\frac{-1}{p-1}}|(\psi_T)_t|^{\frac{p}{p-1}}
			&\leq & C \left(\frac{1}{T}\right)^{\frac{p}{p-1}} T^{\frac{n + 2 + \sum_{j=2}^n \sum_{i =1 }^{j-1} R_{j, i} }{2}}.
		\end{eqnarray*}
		
		Since \begin{equation*}
			\int_{0}^T \int\limits_{\R^2} f \psi_T =  T \int_{0}^1 \Phi_1(t') \int\limits_{\R^2} f \Phi_2(P(x)) = C T \int\limits_{\R^2} f \Phi_2(P(x)),
		\end{equation*}
		we obtain $$\int_{\R^2} f \Phi_2(P(x)) \leq C T^{-1}\int_{0}^T \int_{\R^2} f \psi_T \leq C  T^{\frac{n + \sum_{j=2}^n \sum_{i =1 }^{j-1} R_{j, i} }{2} - \frac{p}{p-1}} .$$
		Using the hypothesis that $p < \frac{n + \sum_{j=2}^n \sum_{i =1 }^{j-1} R_{j, i}}{n - 2 +  \sum_{j=2}^n \sum_{i =1 }^{j-1} R_{j, i}}$ and taking $T \to \infty$ we obtain  $\int_{\R^2} f \leq 0$.
		This contradicts the hypothesis that $\int_{\R^2} f > 0$.
		So, no such $u$ exists and the statement is proved.
	\end{proof}

	\begin{cor}\label{BlowUpTime55}
		Under the terminology and assumptions of \Cref{subcritical51}, assume further that $0 \leq f \in L^1_{loc}(\R^n)$ and  let there be an $\epsilon >0$ such that  $f(x) \geq \epsilon d(x)^{- \lambda}$ whenever $ d(x) \geq 1$ with $ \lambda \in \left(0, \frac{p}{r_1(p-1)}\right)$.
		Let $T_\epsilon$ denote a time such that a weak solution exists in $[0, T_\epsilon)$.
		Then there exists a positive constant $C$ such that $T_\epsilon \leq C \epsilon^{1/\left(r_1 \lambda - \frac{p}{p-1}\right)}$.
	\end{cor}
	\begin{proof}
		The proof follows along the same lines of \Cref{BlowUpTime}, hence the details are skipped.
	\end{proof}
	
	\begin{cor}\label{subcritical1}
		Let $p < \frac{k+2}{k}$, $X_1 = \frac{\partial}{\partial x_1}$ and $X_2 = x_1^k \frac{\partial}{\partial x_2}$ on $\R^2$.
		If $f \in L^1_{loc}(\R^2)$ with $\int_{\R^2} f(x) dx > 0$, then (\ref{MainProblem}) does not admit global in time weak solution.
	\end{cor}
	\begin{proof}
		Although the proof follows directly from \Cref{subcritical55}, one may use  $$d(x) = x_1^{2k+2} + x_2^{2}$$ and  $$f(a) = \frac{a}{T^{k+1}}, \text{ for } a \in \R$$ and follow the same process to obtain the  result.
	\end{proof}
	
	\begin{cor}\label{BlowUpTime1}
		Under the terminology and assumptions of \Cref{subcritical1}, assume further that $0 \leq f \in L^1_{loc}(\R^n)$ and  let there be an $\epsilon >0$ such that  $f(x) \geq \epsilon d(x)^{- \lambda}$ whenever $ d(x) \geq 1$ with $ \lambda \in \left(0, \frac{p}{(k+1)(p-1)}\right)$.
		Let $T_\epsilon$ denote a time such that a weak solution exists in $[0, T_\epsilon)$.
		Then there exists a positive constant $C$ such that $T_\epsilon \leq C \epsilon^{1/\left((k+1) \lambda - \frac{p}{p-1}\right)}$.
	\end{cor}
	\begin{proof}
		The proof follows along the same lines of \Cref{BlowUpTime}, hence the details are skipped.
	\end{proof}

	\begin{cor}\label{subcritical52}
		Let $p < \frac{s_1 + s_2 + s_2 k + k + 3}{s_1 + s_2 + s_2 k + k + 1}$, $X_1 = \frac{\partial}{\partial x_1}$, $X_2 = x_1^k  \frac{\partial}{\partial x_2}$ and $X_3 = x_1^{s_1}x_2^{s_2}  \frac{\partial}{\partial x_3}$ on $\R^3$.
		If $f \in L^1_{loc}(\R^3)$ with $\int_{\R^3} f(x) dx > 0$, then (\ref{MainProblem}) does not admit global in time weak solution.
	\end{cor}
	\begin{proof}
		Although the proof follows directly from \Cref{subcritical55}, one may use  $$d(x) = x_1^{2(1+s_1+s_2+s_2k)} + x_2^{\frac{{2(1+s_1+s_2+s_2k)}}{k+1}} + x_3^2 \text{ and }   f(a) = \frac{a}{T^{1+s_1+s_2+s_2k}}, \text{ for } a \in \R$$ and follw the same process to obtain the result.
	\end{proof}

	\begin{cor}\label{BlowUpTime52}
		Under the terminology and assumptions of \Cref{subcritical52}, assume further that $0 \leq f \in L^1_{loc}(\R^n)$ and  let there be an $\epsilon >0$ such that  $f(x) \geq \epsilon d(x)^{- \lambda}$ whenever $ d(x) \geq 1$ with $ \lambda \in \left(0, \frac{p}{(1+s_1+s_2+s_2k)(p-1)}\right)$.
		Let $T_\epsilon$ denote a time such that a weak solution exists in $[0, T_\epsilon)$.
		Then there exists a positive constant $C$ such that $T_\epsilon \leq C \epsilon^{1/\left((1+s_1+s_2+s_2k)  \lambda - \frac{p}{p-1}\right)}$.
	\end{cor}
	
	\begin{proof}
		The proof follows along the same lines of \Cref{BlowUpTime}, hence the details are skipped.
	\end{proof}

	\begin{remark}
		In the following cases, the method used in this article does not work because no suitable $d$ can be chosen.
		\begin{enumerate}
			\item $X_1 = \frac{\partial}{\partial x_1}$, $X_2 =  x_1^k \frac{\partial}{\partial x_2}$ and $X_3 =  (x_1^2 + x_2^2 + x_1 x_2) \frac{\partial}{\partial x_3}$ on $\R^3$ whenever $k \neq 0$.
			\item $X_1= a(x_2, x_3) \frac{\partial}{\partial x_1}$ and $X_2 = x_1^k \frac{\partial}{\partial x_2}$ on $\R^3$ where $a(x_2, x_3)$ is a homogeneous polynomial and $k \in \N \cup \{0\}$.
		\end{enumerate}
	\end{remark}

	\section*{Funding}
	This research was funded by Nazarbayev University under Collaborative Research Program Grant 20122022CRP1601.
	
	\section*{Author contribution}
	This research is the outcome of discussions between the authors, with both contributing equally to every aspect of the work.
	
	\section*{Conflict of interest}
	The authors declare that they have no conflict of interest.
	
	\section*{Data availability}
	Data sharing does not apply to this manuscript as no datasets were generated or analyzed during the study.

\end{document}